\documentclass{article}
\usepackage{amssymb}
\usepackage{mathtools}
\usepackage{enumitem}
\usepackage{amsthm}
\usepackage{amsopn}
\usepackage{bbm}
\usepackage{dsfont}

\newtheorem{theorem}{Theorem}

\newtheorem{proposition}[theorem]{Proposition}
\newtheorem{remark}[theorem]{Remark}

\newcommand{\vs}{\vskip.075in}
\newcommand{\R}{\mathbb{R}}

\newcommand{\dint}{\displaystyle \int}
\newcommand{\ds}{\displaystyle}
\newcommand{\be}{\begin{equation}}
\newcommand{\ee}{\end{equation}}
\newcommand{\dd}{\mathrm{d}}
\newcommand{\bee}{\begin{equation*}}
\newcommand{\eee}{\end{equation*}}

\begin{document}

\title{On the optimal portfolio problem with partial information and related
mean field games\thanks{%
The work was presented at conferences and seminars at CIRM (Centre
International de Rencontres Mathematiques), IMSI (Institute of Mathematical
and Statistical Innovation) at U. Chicago, Columbia University, University
of Oxford and University of Paris, VII. The authors would like to thank the
participants for their comments and suggestions, and in particular, A.
Capponi, R. Carmona, R. Cont, I. Karatzas, M. Nutz, N. Touzi and X. Y. Zhou.} \thanks {The first author was partially supported by the National Foundation grants
DMS-2153822 and DMS-2452972.}}
\author{Panagiotis E. Souganidis\thanks{%
Department of Mathematics, University of Chicago; souganidis@uchicago.edu} and
Thaleia Zariphopoulou\thanks{%
Departments of Mathematics and IROM, McCombs School of Business, The
University of Texas at Austin, and the Oxford-Man Institute, University of
Oxford; zariphop@math.utexas.edu}}
\maketitle

\begin{center}
\textbf{Abstract}
\end{center}

We study optimal portfolio choice models in markets with partial information about the stock’s drift. 
We first solve the single agent problem for general utilities using a new approach that yields regularity 
of the value function and closed-form expressions for the optimal processes. 
We then consider a $N-$player game in which players interact through the law of peers’ wealth and 
study  its mean field limit.  
This leads to a mean field game  with 
common noise in a reduced,
complete information market with unbounded controls in both the drift and
the volatility. For its solution, we derive a new master system, comprised by
the master equation and an optimality condition for the candidate mean field
equilibrium control. We analyze the cases of separable couplings and general
utilities.  Using insights from indifference valuation, we represent the value of the game as a compilation of the
solution to the problem of a  single player without competition  and a function solving a non-local
quasilinear pde in the space of measures.  When the couplings 
depend  only on the average of  peers' wealth, we derive explicit solutions and various 
regularity results for  representative cases.

\vs

\section{Introduction}

The paper contributes to the areas of optimal portfolio management with
partial information and to mean field games (MFG) in such markets with
common noise and unbounded controls. In the first part of the work, we
revisit the single agent (no competition)\ expected utility problem of
terminal wealth. Under general utilities, we produce regularity and
new representation formulae for the value function, and the optimal wealth and
portfolio processes. The results fill a gap between existing explicit
expressions for special (homothetic)\ utilities and abstract results derived
with general duality arguments. In the second part of the paper, we
introduce a $N-$player game and its continuous limit, where the agents
interact because of relative performance concerns involving the law of
peers' wealth. This leads to a mean field game with 
common noise in a reduced,
complete information market with unbounded controls in both the drift and
the volatility. For its solution, we derive a new master system, comprised by
the master equation and an optimality condition for the candidate mean field
equilibrium control. 
\vs
We analyze the cases of separable couplings and general
utilities. Using insights from indifference valuation, we view the value of the game 
as the writer's value of a claim given by the coupling at the optimum. This guides us to
 represent the value of the game as a compilation of the
solution to the single player problem without competition and a function solving a non-local
quasilinear pde in the space of measures. When the couplings 
depend  only on the average of  peers' wealth, we derive explicit solutions and various 
regularity results for  representative cases.
\vs
The underlying market model consists of a riskless bond (of zero interest
rate)\ and a stock of price $S$ with unknown drift $\Theta $, partially
observed through an observations process $Y$. Classical results from
filtering lead to a reduced, complete information market model with the
stock and the observations process solving%
\begin{equation}\label{S-Y-intro}
\begin{cases}
dS_{s}=b\left( Y_{s},s\right) S_{s}ds+S_{s}dW_{s} \ \ \text{in} \ \ (0,T] \ \ \text{and} \ \ S_{0}=S>0, \\[1.2mm]
dY_{s}=b\left(Y_{s},s\right) ds+dW_{s} \ \ \text{in} \ \ (0,T] \ \ \text{and} \ \ Y_{0}=y \in \mathbb{R}.
\end{cases}
\end{equation}%

The drift process $$b\left( Y_{s},s\right) =\mathbb{E}\left[ \left. \Theta
\right \vert \mathcal{F}_{s}^{Y}\right],$$
is known as  the best estimator $\hat{\Theta}$
given the observations from $Y,$ and the innovation process $W$ is a
Brownian motion defined on $\left[ 0,T\right] $ as 
$$W_{s}=Y_{s}-\int_{0}^{s}\hat{\Theta}_{u}\dd u. $$ %
In this market, we first consider the maximal expected utility problem of a single agent in the absence 
of competition. The state controlled process $\mathcal{X}$ models her 
wealth of the (single) agent and satisfies 
\be\label{takis103}
d\mathcal{X}_{s}=b\left( Y_{s},s\right) \alpha _{s}ds+\alpha _{s}dW_{s} \ \ \text{in} \ \ (t,T] \ \ \text{and} \ \ %
\mathcal{X}_{t}=x\in \mathbb{R},
\ee%
with $\alpha $ being the control policy, representing the amount invested in
stock, which is assumed to be measurable only with respect to $Y;$ we denote 
the set of such controls as $\mathcal{A}$. 
\vs
The value function is
defined as 
\begin{equation}
u\left( x,y,t\right) =\sup_{\alpha \in \mathcal{A}}\mathbb{E}\left[ \left. J(%
\mathcal{X}_{T})\right \vert \mathcal{X}_{t}=x,Y_{t}=y\right] \text{ \ in }%
\mathbb{R\times R\times }\left[ 0,T\right] ,  \label{u-intro}
\end{equation}%
with $J$ being the agent's utility at horizon $T$. 
\vs
We establish that $u$ is
the unique smooth solution of the associated HJB\ equation,

\be \label{HJB-intro}
u_{t} -\frac{\left(
bu_{x}+u_{xy}\right) ^{2}}{2u_{xx}} +\frac{1}{2}u_{yy} +bu_{y} =0 \text{ \ in \ }\mathbb{R\times R\times }\left[ 0,T\right), 
\end{equation}%

with terminal condition
\be\label{takis1.11}
 u(x, y,T)=J(x).
 \ee
The optimal feedback control $a^{\ast }:\R\times\R\times [0,T]\to \R$ is
given by%
\be
a^{\ast }=-\frac{b u_{x}+u_{xy}}{u_{xx}}, 
\label{feedback-intro}
\end{equation}%
%
and satisfies an autonomous equation (see  (\ref{R-equation})). 
\vs
We show that
the optimal wealth and portfolio processes $\mathcal{X}^{\ast \text{ }}$
and $\alpha ^{\ast }$ are represented in the closed-form, 
\begin{equation}
\mathcal{X}_{s}^{\ast \text{ }}=H(H^{(-1)}(x,y,t)+L_{t,s},Y_{s},s) \ \ \text{in} \ \ (t,T] \ \ %
\text{and} \ \ \mathcal{X}_{t}^{\ast \text{ }}=x,
\label{wealth-intro}
\end{equation}%
and%
\begin{equation}\label{alpha-intro}
\begin{split}
&\alpha _{s}^{\ast \text{ }}=c(Y_{s},s)H_{z}(H^{(-1)}(x,y,t)+L_{t,s},Y_{s},s)\\
&\qquad +H_{y}(H^{(-1)}(x,y,t)+L_{t,s},Y_{s},s)\text{ \ in \  } (t, T],
\end{split}
\ee
where the inverse functions are with respect to the $x-$argument. 
\vs
The
process $L$ is defined as 
\begin{equation}\label{L-intro}
L_{t,s}=\int_{t}^{s}b(Y_{\rho },\rho )c(Y_{\rho },\rho )d\rho
+\int_{t}^{s}c(Y_{\rho },\rho )dW_{\rho } \  \ \text{in } \  (t, T] \ \ \text{and} \ \ L_{t,t}=0,
\end{equation}%
and the auxiliary function $H$ solves the linear pde
\begin{equation}
H_{t}+\frac{1}{2}c^{2}H_{zz}+c H_{zy}+\frac{1}{2}H_{yy}=0\text{ \
in }\mathbb{R\times R\times }\left[ 0,T\right) ,  \label{H-intro}
\end{equation}%
with terminal condition
\be\label{takis2}
H(z,y,T)=\left( J^{\prime }\right) ^{\left( -1\right) }(e^{-z}), 
\ee
where  $c:\R\times [0,T]\to \R$ is  given by 
\begin{equation} \label{c-k.1}
c=b+k_{y},
\ee
and $k$ solves 
\be\label{c-k.2}
k_{t}+\frac{1}{2}k_{yy}=\frac{1}{2}b^{2}\  \ \ \text{in } \ \ \mathbb{R\times }\left[ 0,T\right] \ \text{and} \ k(y,T)=0. 
\end{equation}%

It follows that (\ref{wealth-intro})\ yields a closed-form expression for 
the value function, that is, 
\begin{equation}
u(x,y,t)=\mathbb{E}\left[ \left. J\left(
H(H^{(-1)}(x,y,t)+L_{t,T},Y_{T},T)\right) \right \vert X_{t}=x,Y_{t}=y\right].
 \label{u-explicit-intro}
\end{equation}%
To the best of our knowledge, formulae (\ref{wealth-intro}), (\ref{alpha-intro})\ and (%
\ref{u-explicit-intro}), and the regularity results we produce herein are
entirely new when the utility is arbitrary.
\vs
Optimal portfolio problems in high-dimensional complete markets, as the
market model (\ref{S-Y-intro}), have been extensively studied using duality
techniques. While duality is applicable in general semimartingale models, it
is usually difficult, if possible at all, to obtain regularity and
closed-form solutions unless one works with homothetic utilities. In
Markovian models, duality arguments show that the inverse marginal value function satisfies a linear pde,
which follows directly from the HJB
equation and Fenchel transform. On the other hand, as we comment further in
Section~2, the emerging linear pde is not uniformly elliptic, and thus it is
not a priori known if smooth solutions exist. To our knowledge, several
questions in this direction remain open. 
\vs
Herein, we develop an alternative approach that addresses these issues and,
furthermore, produces tractable, closed-form solutions for general
utilities beyond the homothetic ones. Our method  is based on a different
linearization which involves a pair of candidate solutions, and not just a
single equation as in the classical duality theory. Still, one of the emerging
equations is not uniformly elliptic, which we study  by first looking
at the driftless case $\left( b\equiv 0\right)$ and a modified utility. Structurally, the
linearization resembles the one introduced in Musiela and Zariphopoulou \cite{Musiela-Za-monotone} for
an one-dimensional ill-posed HJB equation arising in time-monotone forward
utilities. However, here  we deal with a two-dimensional well-posed linear
pde which requires entirely different arguments. 
\vs

The second part of the paper considers a $N-$player game with relative
performance criteria in the common market (see \eqref{S-Y-intro}). This is
an important area of research in  financial economics and applied finance on optimal
allocation/fund management problems with relative performance concerns, since 
fund management is always performed and evaluated in relation to a benchmark
(index, returns of competitors, clustered financial targets, etc.). 
%
The
prevailing way to classify these models is based on whether agents compete
while investing in a same, common market (asset diversification) or, more
generally, also include individual assets which are inaccessible to their
competitors (asset specialization). For both categories, the existing
applied papers primarily consider only two player games, single period
models, and linear or quadratic criteria (see, among others, \cite{Agarwal}, 
\cite{Basak}, \cite{Basak1}, \cite{Boyle}, \cite{Brown}, \cite{DeMarzo}, \cite{Demarzo2}, 
\cite{Kempf1}, \cite{Kempf2}, \cite{Mitton}, \cite{Sirri} and \cite{Uppal}).
\vs
The literature in continuous time is\textbf{\ }relatively recent. Espinoza and Touzi
introduced in \cite{Espinoza-Touzi} a $N-$player asset
diversification game for players with common exponential utility and linear
competition functions. The work was
extended in a two-player game by  Anthropelos, Geng and Zariphopoulou \cite{Anthropelos-Geng-Z} under forward
criteria.  Lacker and the second author \cite{Lacker-Z.} provided the
first mean field game formulation under asset specialization but, still,
under linear competition and exponential utilities. They, also, studied the
special case of power utilities and geometric competition function, and
non-negative wealth constraints. The MFG in \cite{Lacker-Z.} was defined
probabilistically and, for both cases, static (random) equilibria were
constructed when both common noise (common assets) and individual noise
(specialized assets) were included. Stemming from \cite{Lacker-Z.}, a
substantial literature in MFG\ has been produced allowing, among others, for
intermediate consumption, external habit formation, Ito-diffusion dynamics,
relaxed controls, and mean-variance criteria (see, for example, \cite{bauerle}, \cite{Bo et al},
\cite{Guo-MFG-Regularization}, \cite{Fu-Zhou},  \cite{Kraft}, \cite{Lacker-Soret}, \cite{Zariphopoulou-MFG}   
and others).
\vs

Despite the generality of the market models, all existing works considered
only two pairs of utilities and couplings, specifically, either exponential
utilities and linear couplings (in an infinite wealth domain) or power  utilities and multiplicative couplings (in semi-finite wealth
domain). The only work that incorporated both general utilities and general
couplings on the mean of the peer aggregate wealth  is by the authors in 
\cite{Souganidis-Z-MFG}. Subsequently, in the context of forward criteria, 
\cite{Zariphopoulou-MFG} allowed for general time-monotone forward processes
and linear couplings.
\vs

Another modeling framework that has not been adequately studied, if at all,
in MFG with relative performance is partial information. Frequently, the
drift of the traded asset is not fully known in contrast to the volatility
that is easier to estimate. Portfolio models under partial information (and
in the absence of competition) have been well studied; see, among others, \cite{Bjork-Davis-L}, 
\cite{Brendl}, \cite{Danilova},  \cite{Karatzas-Zhao} and \cite{Lakner}. %
Providing complete bibliography is beyond the scope herein). However, to
our knowledge, with the exception of Deng, Su and Zhou \cite{Deng} and Zhang and Huang \cite{zhang-huang} for
rather special cases, and Huang and Sun \cite{Huang-Sun} for a mean-variance game with
finite players, the work in partial information in MFG with relative
performance is rather limited. MFG with partial
information have been considered by various authors; see, for example, Benoussan and Yam
\cite{Benoussan-Lam}, Sen and Caines \cite{Sen-Caines} and  Shmaya and Zilotto \cite{Shmaya-Zilotto}. In all 
these works, controls were allowed only in the drift and, furthermore,
assumed to be in a compact admissible set. However, optimal portfolio models
have, inherently, controls appearing in both the drift and the volatility,
and, in addition, these controls are in general unbounded. Finally, all works
on MFG with relative performance assume couplings that depend only on the mean
of peer's aggregate wealth (or its geometric analogue when the utility is
power/logarithmic). Herein, we depart from this assumption and allow for
general dependence on the law of peers' wealth. We make all
the above precise next.
\vs
We build on the $N-$player game and its MFG analogue developed in \cite{Souganidis-Z-MFG} 
and introduce their extensions within market models as in  \eqref{S-Y-intro}. 
\vs
The value function of the $i^{th}-$player is defined as the
terminal expected reward,%
\be
\begin{split}
&u^{N,i}\left( x_{1},...,x_{N},y,t\right) =\\
&\sup_{\pi _{i}\in \mathcal{A}}\mathbb{E}\Big[ J\Big( X_{T}^{N,i},%
\dfrac{1}{N-1}\sum \limits_{j=1,j\neq i}^{N}\delta _{X_{T}^{N,j}}\Big)
 \vert X_{t}=(x_{1},...,x_{N}),Y_{t}=y\Big] , \label{v-i}
\end{split} 
\end{equation}%
where  $(x_{1},...,x_{N})\in \mathbb{R}^{N},$ $y\in \mathbb{R}$ and the
payoff function $J$ is assumed to be  the common utility of all players. 
The
controlled state processes $X^{N,i}$ follow  \eqref{takis103} and
the observations process $Y$ is  as in \eqref{S-Y-intro}. 
\vs
The above value
functions are expected to satisfy their related Hamilton Jacobi Bellman
(HJB)\ equations (see  \eqref{HJB-finite}), which in turn lead to a linear system
for the candidate optimal feedback policies (see \eqref{optN}). This
system, however, is non-tractable due to the interlinked dependence of
controls and value functions and their derivatives, for which even
regularity results are lacking. This motivates us to consider the limiting
case, as the number of players goes to infinity.
\vs
Passing formally to the limit, we derive a system of two equations satisfied by the value function 
$U:\R\times\R\times\mathcal{P}_2\times [0,T]\to \R$ and the optimal strategy $\pi^\ast:\R\times\R\times\mathcal{P}_2\times [0,T]\to \R$, where $\mathcal{P}_2$ is the space of probability measures on $\R$ with finite second moment. 
\vs
The first equation is the master equation
\be\label{master-intro}
\begin{split}
&U_{t} +\dfrac{1}{2}  \pi ^{\ast 2}U_{xx}  + \pi ^{\ast } (b U_{x}   +U_{xy} )
+\pi^\ast \mathcal{L}_1U + \mathcal{L}_2 U + b\mathcal{L}_3U + \mathcal{L}_4 U\\[1.5mm]
&+ \frac{1}{2}U_{yy} + bU_{y}=0 \text{ \   in \ }\mathbb{R\times R\times }\mathcal{P}_{2}\times \left[ 0,T\right), 
\end{split}
\ee
with terminal condition
 \be\label{terminal2}
 U(x,y,m,T)=J(x,m),
 \ee
where, for $v:\R\times\R\times \mathcal{P}_{2}\times \left[ 0,T\right)\to \R$, we write
\be\label{takis1000.0}
\mathcal{L}_1v(x,y,m.t)=\int \pi ^{\ast }\left( z,y,m,t\right)
v_{xm}\left( x,y,m,z,t\right) \dd m\left( z\right),
\ee
\be\label{takis1000}
\begin{split}
&\mathcal{L}_2 v(x,y,m,t)=\\[1.5mm]
&\hskip-.1in \dfrac{1}{2}\int \int \pi ^{\ast }\left( z_{1},y,m,t\right) \pi ^{\ast
}(z_{2},y,m,t)v_{mm}\left( x,y,m,z_{1},z_{2},t\right) \dd m\left( z_{1}\right)
\dd m\left( z_{2}\right)\\[1.5mm] 
&+\frac{1}{2}\int \left( \pi ^{\ast }(z,y,m,t)\right) ^{2}v_{mz}\left(
x,y,m,z,t\right) \dd m\left( z\right), 
\end{split}
\ee
\be\label{takis1000.00}
\mathcal{L}_3v(x,y,m.t)=\int \pi ^{\ast }(z,y,m,t)v_{m}\left( x,y,m,z,t\right) \dd m(z),     
\ee
and
\be\label{takis1000.000}
\mathcal{L}_4 v(x,y,m.t)=\int \pi^{\ast }(z,y,m,t)v_{ym}\left( x,y,m,z,t\right) \dd m(z). 
\ee

The second equation of the master system is  the optimality/compatibility condition%
\begin{equation}\label{pi-intro}
\pi ^{\ast } U_{xx} + \mathcal{L}_1 U =-(bU_{x}
+U_{xy}) \text{ \ in \ }\mathbb{%
R\times R\times }\mathcal{P}_{2}\times \left[ 0,T\right].\\
\ee
\vs
The master system is beyond the scope of the current MFG theory; see, Lasry and Lions \cite{Lasry-Lions-I, Lasry-Lions=II},  Lions \cite{Lions}, Huang, Malhame and Caines \cite{Huang-Malhame-Caines},  Carmona and Delarue \cite{Carmona-Delarue} and Cardaliaguet, Delarue, Lasry and Lions \cite{Cardaliaguet}. Indeed, so far,  the existing theory deals with   bounded controls, and, more importantly, 
homogeneous in space and ``uncontrolled'' coefficients  for the  common noise. With all these provisions, in the classical MFG theory the system consisting of \eqref{master-intro}, \eqref{terminal2} and \eqref{pi-intro} collapses to a single infinite dimensional pde. 
\vs

We, in turn, focus on the popular case of separable payoffs, that is, 
\begin{equation}
J(x,m)=J\left( x-C\left( m\right) \right) ,\text{\  \ }  \label{J-payoffs}
\end{equation}%
where  $C:\mathcal{P}_2\to \R$ is, in general,   a function of the law of peers' wealth. 

\vs
We obtain the closed-form solution 
\begin{equation}
U(x,y,m,t)=u(x-f\left( y,m,t\right) ,y,t), 
\label{U-master-intro}
\end{equation}%
to the master equation (\ref{master-intro}) and \eqref{terminal2}, where  $u$ is as in (\ref{u-intro}), that is, the value function of the single agent
problem in the absence of competion. 
\vs
Function  $f:\R\times\mathcal{P}_2\times[0,T]\to \R$ solves  the non-local terminal value problem 
\be\label{f-intro}
f_{t} +\mathcal{L}_2f + \mathcal{L}_4 f 
+\frac{1}{2}f_{yy}  =0\text{ \ in \ }\mathbb{R\times }\mathcal{P}_{2}%
\mathcal{\times }\left[ 0,T\right) \ \text{and} \ f\left( y,m,T\right) =C(m), 
\ee
and the optimality condition (%
\ref{pi-intro}) becomes 
\begin{equation}
\pi ^{\ast } (x,y,m,t)  -\mathcal{L}_1 f (x,y,m,t) 
=f_{y}\left( y,m,t\right) +a^{\ast }(x-f\left( y,m,t\right) ,y,t), 
\ee
where $a^{\ast }$ is the optimal feedback control of the single
player problem; see \eqref{feedback-intro}.
\vs 
The intuition behind representation \eqref{U-master-intro} comes from interpreting the representative agent as the 
writer of claim $C(m^\ast_T)$ (at the optimum).
\vs
We provide details about the connection of the mean field game and an indifference valuation problem of the representative agent in subsection 5.4.
\vs      
If the couplings depend only on the first moment $\bar m=\int z \dd m(z)$ of peers' wealth, that is, for $m\in \mathcal{P}_{2}$,
\[
C(m)=C(\bar{m}), 
\]%
we analyze in detail the cases of exponential utility and arbitrary couplings, and of general utilities and linear couplings.

\vs
For
exponential utilities  $J:\R\to \R$ like $J(x)=-e^{-x}$  and general
couplings, the terminal value problem \eqref{f-intro} becomes autonomous, that is,
\begin{equation*}
\begin{split}
&f_{t} 
+\dfrac{1}{2} \pi^{\ast 2} \int \int f_{\bar{m}\bar{m}}(y,\bar{m},z_{1},z_{2},t) \dd m(z_{1}) \dd m(z_{2})  \\[2mm]
&+ \pi^\ast  \int 
f_{my}(y,\bar{m},z,t)dm(z)  +\frac{1}{2}f_{yy}\left( y,\bar{m},t\right) =0 \ \text {in} \ \mathbb{R\times }\R \mathcal{\times }\left[
0,T\right),\\[2.5mm]  
&f\left( y,\bar{m},T\right) =C(\bar{m}),
\end{split}
\end{equation*}

%
%
%
%
%
%
and 
\eqref{pi-intro} yields the wealth-independent mean field equilibrium
control $\pi ^{\ast }: \mathbb{R\times }\mathcal{P}_{2}\mathcal{\times }%
\left[ 0,T\right] \to \R$,  given by 
\[
\pi ^{\ast }(y,m,t)=\frac{f_{y}\left( y,\bar m,t\right) +c(y,t)}{1-f_{\bar{m}%
}(y, \bar m, t)}. 
\]%
The mean field value is given, for  $(x,y,\bar{m},t) \in \mathbb{R\times R\times }\mathcal{P}_{2}\mathcal{\times }\left[ 0,T\right]$  and  with $k$ as in 
\eqref{c-k.2}, by 
\[
U(x,y,\bar{m},t)=-e^{-(x-f\left( y,\bar{m},t\right) )+k(y,t)}
\]%
and the mean field equilibrium wealth  process $X^{\ast }$ can be written as 
\[
X_{s}^{\ast }=\mathcal{X}_{s}^{x-f\left( y,\bar{m},t\right) ,\ast }+\mathcal{%
E}\left( \int \mathcal{X}_{s}^{x-f\left( y,\bar{m},t\right) ,\ast }\dd m\left(
x\right) ,Y_{s},s)\right) \ \text{in} \ (t,T] \ \text{and} \  X_{t}^{\ast }=x,  
\]%
where 
\[\mathcal{E}\left(y, z, t\right) =f\left(y, 
G(y,z,t),t\right) \ \ \text{and} \ \  G(y,z,t)=\left( \cdot -f\left( y,\cdot
,t\right) \right) ^{\left( -1\right) }\left( y,z,t\right ). \]
\vs 
For general utilities and linear couplings, that is, $C(\bar{m})=\theta \bar{m}$  with $\theta \in \left( 0,1\right)$, 
we obtain 
\[
f\left( y,\bar{m},t\right) =\theta \bar{m}\text{ \ and \ }U(x,y,\bar{m}%
,t)=u(x-\theta \bar{m},y,t). 
\]%
The mean field equilibrium wealth and portfolio processes $X^{\ast }$
and $\pi ^{\ast }$ are given  by 
\[
X_{s}^{\ast }=\mathcal{X}_{s}^{x-\theta \bar{m},\ast }+\frac{\theta }{%
1-\theta }\dint \mathcal{X}_{s}^{x-\theta \bar{m},\ast } \dd m(x) \ \text{in \ } (t,T] \ \text{and} \ X_{t}^{\ast }=x, \]%
and 
\[
\pi _{s}^{\ast }=\mathcal{\alpha }_{s}^{x-\theta \bar{m},\ast }+\frac{\theta 
}{1-\theta }\dint \mathcal{\alpha }_{s}^{x-\theta \bar{m},\ast } \dd m(x)  \ \text{in \ } (t,T], \]%
where $\mathcal{X}^{x-\theta \bar{m},\ast }$ and $\mathcal{\alpha }%
^{x-\theta \bar{m},\ast }$ are the corresponding optimal processes of
the single agent problem, given by \eqref{wealth-intro} and \eqref{alpha-intro} and 
starting at $x-\theta \bar{m}$. 
\vs
The above decomposition
interprets the mean field equilibrium process $X^{x,\ast }$ as the sum of the optimal
wealth $\mathcal{X}^{x-\theta \bar{m},\ast }$ of the single agent starting
at $x-\theta \bar{m}$ plus its mean  in terms of the initial
distribution of $x$ weighted by the factor $\frac{\theta }{1-\theta }.$
\vs
The mean
field equilibrium policy $\pi _{s}^{x,\ast }$ has an analogous decomposition.
\vs
\textbf{Organization of the paper.} The paper is organized as follows. In section~2, we revisit the market model
under partial information and provide the new, alternative solution approach
and the closed-form solutions for general utilities for the single agent without competition problem.  In section~3, we
introduce the $N-$player game and its continuum limit, and derive the master
system for general utilities and couplings. In section~4, we study the class
of separable payoffs and analyze representative cases in section~5. We conclude in
section~6.
\vs

\section{The portfolio choice problem under partial information for single agent without competition%
}

We revisit the optimal portfolio choice problem for single agent without competition in a finite horizon with
general utilities, and with partial information for the stock's drift. We
introduce an alternative solution approach, establish regularity for the
value function, and produce closed-form expressions for the optimal wealth
and portfolio processes. We also present representative examples which cover
the so-called SAHARA\ utilities and utlities with completely monotonic
inverse marginals (CMIM).  

\subsection{The market model and background results on filtering}

We consider a probability space $\left( \Omega ,\mathcal{F},\mathbb{P}%
\right) $ which supports a Brownian motion $B$ with filtration $\mathbb{F}%
^{B}=\left \{ \mathcal{F}_{t}^{B};\text{ }0\leq t\leq T\right \} $ and a
random variable $\Theta :\Omega \rightarrow [\theta_1, \theta_2],$ with $0<\theta_1<\theta_2<\infty$, 
such that 
\be\label{ass1}
\begin{cases}
 \text{$\Theta$ is  independent of $
B$ under $\mathbb{P}$  and has prior probability}\\
 \text{distribution}  \ \nu (A)=\mathbb{P}\big[\Theta \in A\big] \; \text{ for} \ 
 A\in \mathcal{B}\left( \mathbb{R}^+\right),\\
\text{such that,  for each $y\in  \mathbb{R},$}
\dint_{\theta _{1}}^{\theta_{2}}e^{y\theta }\dd \nu(\theta) <\infty. 
\end{cases}
\ee

We consiser  the \text{observations process} 
\begin{equation}
Y_{t}=\Theta t+B_{t} \ \text{ in } \ \left( 0,T\right] \ \ \text{ and } \ \ %
Y_{0}=y\in \mathbb{R},  \label{Y-original}
\end{equation}%
and introduce the $\mathbb{P}$-augmentation $\mathbb{F}^{Y}=\left \{ \mathcal{F}_{t}^{Y};\text{ }0\leq t\leq
T\right \} $  of 

$\mathcal{F}^{Y}=\left \{ \sigma
\left( Y_{s}\right) ;0\leqslant s\leqslant t\leq T\right \} ,$ and the $\mathbb{P}$%
-augmentation $\mathbb{G}=\left \{ \mathcal{G}_{t};0\leq t\leq T\right \} $ of the enlarged filtration  generated by both $\Theta $ and $B,
$%
\[
\mathcal{G}_{t}^{\Theta ,B}=\left \{ \sigma \left( \Theta ,B_{s}\right)
;0\leqslant s\leqslant t\leq T\right \} =\sigma \left( \Theta \right) \vee 
\mathcal{F}_{t}^{B}.
\]

We assume   a financial market with  two securities. The first is  a riskless bond taken
to be the numeraire and offering zero interest rate. The second security is
a risky stock whose (discounted by the numeraire) price $S$  satisfies 
\textbf{\ }%
\begin{equation}
dS_{t}=\Theta S_{t}dt+S_{t}dB_{t}=S_{t}dY_{t} \ \text{in} \ (0,T]  \ \text{and} \ \ S_{0}>0.
\label{stock-theta}
\end{equation}%
The key ingredient is that one cannot observe directly either $%
\Theta $ or the Brownian motion $B,$ but only the levels of process $S$.  In
other words, it is only the information $\mathbb{F}^{Y}$ generated by the
observations process $Y$ that is accessible to the agent.
\vs
In this market, the set of admissible strategies $\mathcal{A}$ consists of
self-financing investment strategies $\pi $ representing the (discounted)
amount allocated in the stock and measurable with respect to the information
generated by the observations process, that is,
\begin{equation}
\mathcal{A}:=\big\{ \pi :\text{ }\pi _{t}\in \mathcal{F}_{t}^{Y}\text{ \  \
and  \ \ }\mathbb{E}   \int_{0}^{T}\pi _{t}^{2}dt<\infty  \big \} .  \label{admissible-set}
\end{equation}

\begin{remark}
Sde  \eqref{stock-theta}  may have the more general form 
\[
dS_{s}=\Theta S_{s}dt+\sigma S_{s}dB_{s} \ \text{in} \ (0,T] \ \ \text{and} \ \ %
S_{0}>0, 
\]%
where $\sigma $ is a known (positive) coefficient, representing the
volatility of the stock. To ease the presentation, we assume $\sigma \equiv
1,$ as the general case can be directly incorporated by a simple rescaling.
\end{remark}

Next, we review some classical results from filtering which yield a reduced 
complete information model; see, among others, \cite{Bain-Crisan}, \cite{Bjork-Davis-L}, 
\cite{Kallianpur} and \cite{Karatzas-Zhao}.
\vs

Let $F:\mathbb{R\times }\left[ 0,T\right] \rightarrow \mathbb{R}^{+}$  and $b:\mathbb{R\times }\left[ 0,T\right] \rightarrow \mathbb{R}^{+}$ be defined  as 
\begin{equation}\label{F-dfn}
F(y,t)=\int_{\mathbb{\theta }_{1}}^{\theta _{2}}e^{y\theta -\frac{1}{2}%
\theta ^{2}t} \dd \nu \left( \theta \right) \ \ \text{and} \ \ b=\frac{F_{y}}{F}. 
\end{equation}%
It follows that $F$ solves the (ill-posed) heat equation 
\begin{equation}
F_{t}+\frac{1}{2}F_{yy}=0 \ \ \text{in} \ \ \R\times (0,T] \ \ \text{ and  } \ \ F(y,0)=\int_{\mathbb{\theta }%
_{1}}^{\theta _{2}}e^{y\theta }d\nu \left( \theta \right) ,
\label{F-equation}
\end{equation}%
 is strictly convex in $y$ and strictly decreasing in $t$, and satisfies, for some constants $c_1=c_1(\nu (\theta )), c_2=c_2(\nu (\theta ))>0$  and all $y\in\R$, %
\begin{equation}
c_1e^{-c_1\left \vert y\right \vert }\leq F(y,T)\leq c_2e^{c_2\left \vert y\right \vert }. 
\label{F-growth}
\end{equation}%
Finally, $F$  is absolutely monotonic, 
which yields that 
\begin{equation}
b_{y}=\frac{F_{yy}F-\left( F_{y}\right) ^{2}}{
F ^{2}}>0 \ \text{in} \ \mathbb{R\times }\left[ 0,T\right].\text{%
\ }  \label{b-derivative}
\end{equation}%
It follows that 
\begin{equation}
0<\theta _{1}\leq b\leq \theta _{2}\  \  \  \text{and \ }\
0<b_{y}\leq \theta _{2}^{2}-\theta _{1}^{2}\text{   \ in } \ \mathbb{%
R\times }\left[ 0,T\right] .  \label{b-bounds}
\end{equation}%
The best estimator of $\Theta $ is given by 
\[
\hat{\Theta}_{t}=\mathbb{E}\left[ \left. \Theta \right \vert \mathcal{F}%
_{t}^{Y}\right] =\left \{ 
\begin{array}{c}
b(Y_{t},t) \ \ \text{if} \ \ t\in \left( 0,T\right], \\ 
\\ 
\int_{\mathbb{\theta }_{1}}^{\theta _{2}}\theta \nu \left( d\theta \right) \ \ 
\text{ if} \ \ t=0,%
\end{array}%
\right. 
\]%
and the so-called \text{innovations process} $W$, defined for $t\in \left[ 0,T\right]$, by 
\begin{equation}
W_{t}=Y_{t}-\int_{0}^{t}\hat{\Theta}_{s}ds=Y_{t}-\int_{0}^{t}b(Y_{s},s)ds,%
 \label{W-innovations}
\end{equation}%
is a $\left( \mathbb{F}^{Y},\mathbb{P}\right) $ standard Brownian motion.
\vs

In view of the above, the original sdes \eqref{stock-theta} and \eqref{Y-original} can be
now written as

\begin{equation}
dS_{t}=b\left( Y_{t},t\right) S_{t}dt+S_{t}dW_{t} \ \text{in} \  (0,T] \ \ \text{ and } \ \ S_{0}=S>0,
\label{stock-reduced}
\end{equation}%
and%
\begin{equation}
dY_{t}=b\left( Y_{t},t\right) dt+dW_{t} \ \text{in} \  (0,T] \ \ \text{ and } \ \ Y_{0}=y\in \mathbb{R}.
\label{Y-reduced}
\end{equation}%
The  model consisting of \eqref{stock-reduced} and  \eqref{Y-reduced} 
constitutes a
complete information market model with a single local factor $Y$.
\vs
For each $\alpha \in \mathcal{A}$, the agent's wealth, denoted by $\mathcal{X},$ solves 
\be\label{takis2000}
d\mathcal{X}_{s}=b\left( Y_{s},s\right) \alpha _{s}ds+\alpha _{s}dW_{s} \ \ 
\text{in} \ \ (t,T] \ \ \text{and} \ \ \mathcal{X}_{t}=x\in \mathbb{R}.\ee
Both $\mathcal{X}$
and $\alpha $ are expressed in discounted by the bond units and are, thus,
unitless quantities. This property is important for the proper definitions
of the payoffs in the upcoming MFG as they are, in general,  allowed to depend on
wealths, their laws, means, quantiles and higher moments.
\vs

The value function $u: \mathbb{%
R\times R\times }\left[ 0,T\right] \to \R$ is defined as 
\begin{equation}
u(x,y,t)=\sup_{\pi \in \mathcal{A}}\mathbb{E}\left[ \left. J(\mathcal{X}%
_{T})\right \vert \mathcal{X}_{t}=x,Y_{t}=y\right].
\label{u-DFN}
\end{equation}%
We introduce next the main assumptions for the
utility function $J:\R\to \R,$  the risk tolerance $r:\R\to \R^+$  and  the inverse marginal utility $I:\R^+\to \R,$ defined  respectively by 
\begin{equation}
r=-\frac{J^{\prime }}{J^{\prime \prime }} \ \  \text{and} \ \  I=\left( J^{\prime }\right) ^{\left( -1\right) }. \label{risk}
\end{equation}%

We assume that 
\be\label{utility}
\begin{split}
&J \text{ \ is a strictly concave and
strictly increasing map in }\mathcal{C}^{4}\left( \mathbb{R}\right), \ \text{and} \\[1.5mm]
&\text{there exist $A,B, c>0$ and $\gamma >1$ such that, for all $x\in \R$,}\\[1.5mm]
&\qquad B \leq r(x)\leq \sqrt{Ax^{2}+B}, \ \ \left
\vert r^{\prime } \right \vert \leq A \ \ \text{and} \ \ 
\underset{z\to \infty} \lim z^\gamma I(z)=c.
\end{split}
\ee



The  assumptions above are mild and are  satisfied by a large class of utility
functions like, for example, the exponential utility 
\be\label{exp}
J(x)=-C_{1}e^{-B
x}+C_{2} \ \ \text {for some $C_{1} >0$ \ and \ $C_{2}\in \mathbb{R},$}
\ee
for which the risk tolerance is constant $r(x)=B,$  and
by the so-called SAHARA\ utilities, introduced in \cite{Musiela-Za-monotone}
(see, also, \cite{Chen-SAHARA} and \cite{zariphopoulou-zhou}), that are
represented via their risk tolerances in the parametric form%
\begin{equation}
r(x)=\sqrt{Ax^{2}+B} \ \ \text{for} \ \ A,B>0.  \label{double}
\end{equation}%
They, also, apply to the extended family of utilities with completely
monotonic inverse marginals (CMIM),\ introduced in \cite{Musiela-Za-monotone}
and further examined in \cite{Mostovi-et-al}. We further elaborate on these
utilities in subsection 2.3.
\vs

We note that (\ref{risk})  implies that 
\begin{equation}
-zI^{\prime }(z)=r(I(z))\text{ \  \ and \  \ }zI^{\prime }(z)+z^{2}I^{\prime
\prime }(z)=r(I(z))r^{\prime }(I(z)).
\label{I-identities}
\end{equation}%

It is not surprising that, since  the market
dynamics (\ref{stock-reduced}) and (\ref{Y-reduced})\ compose a complete
market model,  
(\ref{u-DFN})\ can be linearized.
Such linearization is inherent in the duality approach and, in Markovian
models, a linear pde can be easily derived from the HJB equation and Fenchel
transform. However, for general utilities, this pde is two-dimensional and
not uniformly elliptic and, thus, regularity results are not a priori known.
To the best of our knowledge, this issue has not been addressed in the
existing literature, except for the case of homothetic utilities where the
HJB\ equation reduces to a simple one-dimensional pde. 
\vs

Here, we circumvent this issue by first looking at the  auxiliary problem
 \eqref{u-tilda}  for  which the related linear pde \eqref{heat-h} is one-dimensional
and, thus, regularity results are easy to obtain. We
stress that this dimensionality reduction works not because of utility
homotheticity. Rather it results from  suitably modifying the utility and, furthermore, removing the drift in the wealth
dynamics,  which essentially results in  the optimal wealth process being 
static in the $x-$argument, as manifested in \eqref%
{optimal-processes-tilda}.
\vs

A similar optimization problem was considered in \cite{Karatzas-Zhao} but in
a different context within the duality approach employed therein. Here,
problem (\ref{u-tilda}) plays a different role. Firstly, it helps us
establish regularity for the original problem under much milder and direct
assumptions on the utility function. Secondly, the expressions (see \eqref%
{optimal-processes-tilda} below) for the optimal processes guide us how to
construct their general analogues, (\ref{x-thm}) and (\ref{a-process-thm}),
whose form turns out to be quite useful in the MFG\ we examine later on.

\subsection{An auxiliary expected utility problem}

We consider a fictitious portfolio choice problem with modified terminal
utility and dynamics. Specifically, let $W^{\mathbb{Q}}$ be a Brownian
motion on probability space $\left( \Omega ,\mathcal{F},\mathbb{Q}\right) $
and, for $t\in [0,T]$,  consider the processes $\tilde{\mathcal{X}}$ and ${\mathcal{Y}}$ solving
\begin{equation} \label{X-Y-tilda}
\begin{split}
&d\tilde{\mathcal{X}}_s=\tilde {\alpha}_s dW^{\mathbb{Q}} \ \text{in} \ (t,T] \ \text{and} \ \tilde{\mathcal{X}_t}=x\in \R,\\[1.5mm]
& d \mathcal{Y}_s=dW^{\mathbb{Q}} \ \text{in} \ (t,T] \ \text{and} \ \mathcal{Y}_t=y \in \R.
\end{split}
\ee

The value function  $w:\R\times\R\times [0,T]\to \R$ is defined  by 
\begin{equation}
w(x,y,t)=\sup_{\tilde{\alpha}\in \tilde{ \mathcal{A}}}\mathbb{E}_{\mathbb{Q}}\left[
\left. J\left( \mathcal{\tilde{X}}_{T}\right) F(\mathcal{Y}_{T},T)\right
\vert \mathcal{\tilde{X}}_{t}=x\text{, }\mathcal{Y}_{t}=y\right], \label{u-tilda}
\end{equation}%
with $F$ as in (\ref{F-dfn}) and $\tilde{\mathcal{A}}$ defined similarly to  \eqref{admissible-set}.
\vs

To  find  the value function $w(x,y,t)$ and construct the optimal processes 
$\mathcal{\tilde{X}}^{\ast }$ and $\tilde{\alpha}^{\ast }$,  we
first introduce a key auxiliary function and study its properties.

\begin{proposition}
(i) There exists a unique smooth solution $h:\R\times\R\times[0,T]\to \R$ to  the terminal value problem 
\begin{equation}
h_{t}+\frac{1}{2}h_{yy}=0 \ \text{in} \ \R\times [0,T] \ \ \text{and}   \  \ h(z,y,T)=I\left( \frac{e^{-z}}{F(y,T)}%
\right),
\label{heat-h}
\end{equation}
with $I$ and $F$ as in (\ref{risk}) and (\ref{F-dfn}) respectively, such that, for each \\ $\left( y,t\right) \in \mathbb{R\times 
}\left[ 0,T\right], $
\be\label{heat-1}
\begin{split}
&\lim_{z\rightarrow -\infty }h(z,y,t)=-\infty \ \text{ and} \  \lim_{z\rightarrow
\infty }h(z,y,t)=\infty,\\[1.2mm]
&\hskip.075in h_{y}>0\text{ \ and \ }h_{z}>0\text{ \ in \ }\mathbb{R\times
R\times }\left[ 0,T\right], 
\end{split}
\end{equation}%
and, for each $\left( y,t\right) $ in $\mathbb{R\times }\left[ 0,T\right] $, the
inverse in $z$ function $h^{\left( -1\right) }(\cdot, y,t):\R\to \R$ exists and is strictly increasing.
\vs
(ii) For all $(x,y,t) \in \mathbb{R\times R\times }\left[ 0,T\right]$ and the 
 constants $A$ and $B$  in \eqref{utility}, 
\begin{equation}
h_{y}(z,y,t)\leq \theta _{2}\sqrt{Ah^{2}\left( z,y,t\right) +Be^{A(T-t)}}, %
\label{heat-2}
\end{equation}%
and 
\begin{equation} \label{heat-3}
-Ah_{z}\left( z,y,t\right) \leq h_{yz}(z,y,t)\leq Ah_{z}\left( z,y,t\right).
\end{equation}%
\end{proposition}

\begin{proof}
(i) The smoothness and uniqueness results follow from  (\ref{F-dfn}), \eqref{F-growth} and (\ref{utility}).
\vs
Since 
\begin{equation}
h_{y}(z,y,T)=-\frac{F_{y}(y,T)}{F(y,T)}\frac{e^{-z}}{F(y,T)}I^{\prime
}\left( \frac{e^{-z}}{F(y,T)}\right) >0\text{,\ }  \label{h-y-z}
\end{equation}%
and%
\begin{equation}
h_{z}(z,y,T)=-\frac{e^{-z}}{F(y,T)}I^{\prime }\left( \frac{e^{-z}}{F(y,T)}%
\right) >0, \label{h-z-z}
\end{equation}%
the  maximum principle yields (\ref{heat-1}).\ 
\vs


(ii) At $t=T,$ (\ref{h-y-z}) combined with (\ref{b-bounds}), (\ref{utility}) and (\ref{I-identities}), and the terminal condition in (\ref%
{heat-h}) imply that 
\[
h_{y}(z,y,T)=b(y,T)r\left( I\left( \frac{e^{-z}}{F(y,T)}\right) \right) \leq \theta _{2}\sqrt{Ah^{2}\left( z,y,T\right) +B}.\]
%
Let $g:\R\times\R\times[0,T]\to \R^+$ and $C: \R\times\R\times[0,T]\to \R$ be defined by 
\begin{equation}\label{g-aux}
g(z,y,t)=\theta _{2}\sqrt{Ah^{2}\left( z,y,t\right) +Be^{A(T-t)}},
\end{equation}%
and
\bee
C(z,y,t)=\frac{ABe^{A(T-t)}}{2g^{3}}\left( g+h_{y}\right),
\eee
and note that 
\begin{equation}
h_{y}(\cdot,\cdot,T)\leq g(\cdot, \cdot,T) \ \text{in}  \ \R\times\R,  \label{h-y-g-terminal}
\end{equation}%
and, in view of \eqref{g-aux} 
%
%
and the first inequality in (\ref{heat-1}), 
$C>0 \ \text{in} \ \mathbb{R\times R\times }\left[ 0,T\right].$
\vs
We claim that $h_{y}-g$ solves 
\begin{equation}
\left( h_{y}-g\right) _{t}+\frac{1}{2}\left( h_{y}-g\right)
_{yy}+C\left( h_{y}-g\right) =0\text{ \ in \ }\mathbb{R\times R\times 
}\left[ 0,T\right) \text{.}  \label{h-g}
\end{equation}%
To ease the notation, we assume that $\theta _{2}=1$. Then
direct calculations yield that 
\begin{equation*}
g_{t}+\frac{1}{2}g_{yy} = \frac{ABe^{A(T-t)}}{2g^{3}}%
\left( g^{2}-\left( h_{y}\right) ^{2}\right).
\eee
Using
(\ref{h-y-g-terminal})\ and (\ref{h-g}) and applying  the maximum principle we find that  
\[
h_{y}\leq g \text{  \ in \ }\mathbb{R\times R\times }\left[ 0,T%
\right] , 
\]%
and, thus,  (\ref{heat-2})\ holds.
\vs
Finally, using again the maximum principle, we note that  to  show (\ref{heat-3}), it suffices to establish 
it for $t=T.$ 
\vs
To this end, note that  \eqref{F-dfn} and (\ref{h-y-z}) give 
\[
h_{yz}(z,y,T)=b(y,T)h_{zz}(z,y,T), 
\]%
and, thus, after using (\ref{I-identities}) and (\ref{h-z-z}), 
\[
h_{yz}(z,y,T)=b(y,T)h_{z}\left( z,y,T\right) r^{\prime }\left( I\left( \frac{e^{-z}}{%
F(y,T)}\right) \right).
\]
The latter equality, (%
\ref{b-bounds}) and the derivative estimate in (\ref{utility}) give 
\[
\left \vert h_{yz}(z,y,T)\right \vert =b(y,T)h_{z}\left( z,y,T\right)
\left
\vert r^{\prime }\left( I\left( \frac{e^{-z}}{F(y,T)}\right) \right)
\right
\vert \leq Ah_{z}\left( z,y,T\right) , 
\]%
and we conclude.

\end{proof}

Next, we construct the value function $w$ and the optimal processes $%
\mathcal{\tilde{X}}^{\ast }$ and \ $\tilde{\alpha}^{\ast }$.

\begin{proposition}
Assume that   $\mathcal{Y}$ solves (\ref%
{X-Y-tilda}). Then:
\vs
(i)~The value function $w$ defined in  \eqref{u-tilda} is the unique  in  $\mathcal{C}^{2,2,1}\left( \mathbb{R\times
R\times }\left[ 0,T\right] \right) $ solution  to the terminal value HJB\ equation%
\be\label{HJB-tilda}
\begin{split}
&w_{t}+\max_{\alpha }\left( \frac{1}{2}\alpha ^{2}w_{xx}+\alpha w_{xy}\right)
+\frac{1}{2}w_{yy} \\
&=w_{t}-\frac{\left( w_{xy}\right) ^{2}}{2w_{xx}}+\frac{1}{2}w_{yy}=0\text{ \
\ in \ }\mathbb{R\times R\times }\left[ 0,T\right),\\[1.5mm]
&w(x,y,T)=J\left( x\right) F(y,T),
\end{split}
\end{equation}%
and is given, for $h$ as in (\ref{heat-h}),   by 
\begin{equation}\label{u-tilda-representation}
w(x,y,t)=\mathbb{E}_{\mathbb{Q}}\left[ \left. J(h\left( h^{\left( -1\right)
}\left( x,y,t\right) ,\mathcal{Y}_{T},T\right) F(\mathcal{Y}_{T},T)\right
\vert \mathcal{Y}_{t}=y\right]. 
\ee
Furthermore, for each  $(z,y,t) \in \R\times \R\times \left[ 0,T\right]$,  
\begin{equation}\label{w-h-ez}
w_{x}\left( h\left( z,y,t\right) ,y,t\right) =e^{-z}.   
\end{equation}
\vs
(ii)~For each $(x,y,t) \in \R\times \R\times \left[ 0,T                     %
\right]$, the optimal feedback policy is given  by  
\be\label{a-optimal-tilda}
\tilde{\alpha}^{\ast }(x,y,t)=-\frac{w_{xy}\left( x,y,t\right) }{
w_{xx}\left( x,y,t\right) }=h_{y}\left( h^{( -1)}(x,y,t),y,t \right),
\end{equation}
and satisfies, for each $(x,y,t) \in \R\times\R\times \left[ 0,T                     %
\right]$, the inequalities%
\begin{equation}\label{a-estimate}
0<\tilde{\alpha}^{\ast }(x,y,t)\leq \theta _{2}\sqrt{Ax^{2}+Be^{A\left(
T-t\right) }} \ \ \text{and} \  \  \tilde{\alpha}_{x}^{\ast }(x,y,t)\leq \theta
_{2}A. 
\end{equation}
\vs
(iii)~The optimal processes $\mathcal{\tilde{X}}^{\ast }$ and $\tilde{\alpha}%
^{\ast }$ are given, respectively, for each $(x,y) \in \R\times \R$ and $0\leq t \leq s \leq T$,  by 
\begin{equation}
\mathcal{\tilde{X}}_{s}^{\ast }=h\left( h^{\left( -1\right) }\left(
x,y,t\right) ,\mathcal{Y}_{s},s\right) \ \text{and } \ \tilde{\alpha}_{s}^{\ast
}=h_{y}\left( h^{\left( -1\right) }\left( x,y,t\right) ,\mathcal{Y}%
_{s},s\right). 
\label{optimal-processes-tilda}
\end{equation}
\end{proposition}

\begin{proof}
(i)\ A straightforward modification of the results in \cite{Zariphopoulou-MOR} shows that the
value function (\ref{u-tilda}) is the unique viscosity solution of (\ref%
{HJB-tilda}) in the class of functions that are strictly concave and
strictly increasing in $x$, for each $\left( y,t\right) \in \mathbb{R\times }%
\left[ 0,T\right] $. 
\vs 
We now construct a smooth solution $\hat{w}$ of (\ref%
{HJB-tilda}) in the same class of functions which, by uniqueness, will
coincide with $w$. 
\vs
To this end, consider the terminal value problem 
\begin{equation}
\begin{split}
&K_{t}+\frac{1}{2}K_{yy}=0\  \  \text{in } \ \mathbb{R\times R\times }\left[ 0,T%
\right),\\[1.5mm]
&  K(z,y,T)=J\Big( I\big( \dfrac{e^{-z}}{F(y,T)}\big) 
\text{\ }\Big) F(y,T),  \label{heat-K}
\end{split}
\end{equation}%
which, in view of  (\ref{F-dfn}), \eqref{F-growth} and (\ref{utility}),  is
well-posed and has a unique smooth solution.
\vs
Using the terminal condition in (\ref{heat-h}), we rewrite 
\begin{equation}
K(z,y,T)=J\left( h(z,y,T)\right) F(y,T),  \label{F-G-tilda}
\end{equation}%
and find 
\begin{equation*}
K_{z}(z,y,T)=h_{z}(z,y,T)J^{\prime }\left( I\left( \frac{e^{-z}}{F(y,T)}\right) \right)
F(y,T)=e^{-z}h_{z}(z,y,T). 
\end{equation*}
In addition, for each fixed $z\in \mathbb{R},$ functions $K_{z}(z,y,t)$ and $%
e^{-z}h_{z}(z,y,t)$ solve in $\mathbb{R\times }\left[ 0,T\right) $ the same
heat equation (see (\ref{heat-h}) and (\ref{heat-K})) and, thus, by
uniqueness, for each $(x,y,t)\in  \mathbb{R\times R\times }\left[
0,T\right]$, 
\begin{equation} \label{F-h-tilda}
K_{z}(z,y,t)=e^{-z}h_{z}(z,y,t).
\end{equation}%
Next, we introduce the smooth  function $\hat{w}: \mathbb{R\times R\times }\left[ 0,T\right] \to \R$ given 
by 
\begin{equation}\label{u-F-tilda}
\hat{w}(x,y,t)=K\Big( h^{\left( -1\right) }(x,y,t),y,t\Big), 
\end{equation}%
with $h^{\left( -1\right) }$ as in Proposition 2.
\vs 
We show  that, for each $\left( y,t\right) \in \mathbb{R\times }\left[ 0,T\right] $,   $\hat w$ is strictly increasing and strictly concave in $x,$ and, in addition, it 
satisfies (\ref%
{HJB-tilda}).
\vs 

To
ease the notation, we write 
\begin{equation}\label{m-h-inverse}
p(x,y,t)=h^{\left( -1\right) }(x,y,t).
\end{equation}%
Then, for each $(x,y,t)\in \R\times\R\times[0,T]$, 
\begin{equation}
p_{t}=-\frac{h_{t}(p,y,t)}{h_{z}(p,y,t)},\text{ \  \ }p_{x}=\frac{1}{%
h_{z}(p,y,t)}  \label{m-1}
\end{equation}%
\begin{equation}
p_{xx}=-\frac{h_{zz}(p,y,t)}{\left( h_{z}(p,y,t)\right) ^{3}},\text{ \ }%
p_{y}=-\frac{h_{y}(p,y,t)}{h_{z}(p,y,t)}  \label{m-2}
\end{equation}%
\begin{equation} \label{m-3}
\begin{split}
&p_{yy}=-\frac{1}{h_{z}\left( p,y,t\right) }\left( \Big( \frac{h_{y}(p,y,t)}{%
h_{z}(p,y,t)}\Big) ^{2}h_{zz}(p,y,t)\right. \\ 
&\qquad \left. -2\frac{h_{y}(p,y,t)}{h_{z}(p,y,t)}h_{zy}(p,y,t)+h_{yy}(p,y,t)\right)
\end{split}
\end{equation} 
and 
\begin{equation}
p_{xy}=-\frac{1}{\left( h_{z}\left( p,y,t\right) \right) ^{2}}\left( -\frac{%
h_{y}(p,y,t)}{h_{z}(p,y,t)}h_{zz}(p,y,t)+h_{zy}(p,y,t)\right) ,  \label{m-4}
\end{equation}%

It follows from (\ref{F-h-tilda}) and
(\ref{u-F-tilda}) that 
\[
\hat{w}_{t}=p_{t}K_{z}(p,y,t)+K_{t}(p,y,t), 
\]%
and 
\begin{equation}
\hat{w}_{x}=p_{x}K_{z}(p,y,t)=\frac{K_{z}(p,y,t)}{h_{z}(p,y,t)}=e^{-p}.
\label{v-monotone-tilda}
\end{equation}%
Furthermore, 
\begin{equation}
\hat{w}_{xx}=-p_{x}\hat{w}_{x}\text{,\ }\tilde{w}_{xy}=-p_{y}\tilde{w}_{x}%
\text{ \ and \ }\tilde{w}_{y}=p_{y}K_{z}(p,y,t)+K_{y}(p,y,t),\text{\ }
\label{v-concave-tilda}
\end{equation}%
and 
\[
\tilde{w}_{yy}=p_{yy}K_{z}(p,y,t)+\left( p_{y}\right)
^{2}K_{zz}(p,y,t)+2p_{y}K_{zy}(p,y,t)+K_{yy}(p,y,t). 
\]%
Next, we use that $K$ solves (\ref{heat-K}) and  show that $\hat{w}$ satisfies (\ref{HJB-tilda}). Indeed, 
\begin{equation*}
\begin{split}
& \hat{w}_{t}-\frac{1}{2}\frac{\left( \hat{w}_{xy}\right) ^{2}}{\hat{w}_{xx}}+%
\frac{1}{2}\hat{w}_{yy} \\
&=\left( p_{t}+\frac{1}{2}p_{yy}\right) K_{z}(p,y,t)+\frac{1}{2}\frac{\left(
p_{y}\right) ^{2}}{p_{x}}\hat{w}_{x}+\frac{1}{2}\left( p_{y}\right)
^{2}K_{zz}(p,y,t)+p_{y}K_{zy}(p,y,t). 
\end{split}
\end{equation*}
In addition, from (\ref{F-h-tilda}) we 
 deduce that
\[
K_{zz}=-e^{-z}h_{z}+e^{-z}h_{zz}\text{ \  \ and \ } \ K_{zy}=e^{-z}h_{zy}.%
\]%
Combining the above and using (\ref{v-monotone-tilda}), we find, after using (\ref{heat-h}), that%
\begin{equation*}
\hat{w}_{t}-\frac{1}{2}\frac{\left( \hat{w}_{xy}\right) ^{2}}{\hat{w}_{xx}}+%
\frac{1}{2}\hat{w}_{yy}  =-\big( h_{t}+\frac{1}{2}h_{yy}\big) h_{z}=0.
\end{equation*}

Finally, we deduce from (\ref{v-monotone-tilda}) that $\hat{w}_{x}(x,y,t)>0$.  
Then  (\ref{v-concave-tilda}) together with the monotonicity of $%
h^{\left( -1\right) }$\ yield $\hat{w}_{xx}(x,y,t)<0$. 
We easily conclude
that $\hat{w}\equiv w$ in $\mathbb{R\times R\times }\left[ 0,T\right] $.
\vs

To show (\ref{u-tilda-representation}), we first observe that (\ref{heat-K}%
)\ yields a probabilistic representation for $K$, that is, 
\bee
\begin{split}
K(z,y,t)&=\mathbb{E}_{\mathbb{Q}}\left[ \left. J\left( I\left( \frac{e^{-z}}{%
F(\mathcal{Y}_{T},T)}\right) \text{\ }\right) F(\mathcal{Y}%
_{T},T)\right
\vert \mathcal{Y}_{t}=y\right]\\[1.5mm]
&=\mathbb{E}_{\mathbb{Q}}\left[ \left. J\left( h\left( z,\mathcal{Y}%
_{T},T\right) \right) F(\mathcal{Y}_{T},T)\right \vert \mathcal{Y}_{t}=y%
\right] , 
\end{split}
\eee
where the last equality follows from the terminal condition in (\ref{heat-h}%
). Then, (\ref{u-tilda-representation})  and \eqref{w-h-ez} follow respectively  from (\ref{u-F-tilda}),
using that $\hat{w}\equiv w$ and  (\ref{v-monotone-tilda}).
\vs
(ii) The first order conditions in (\ref{HJB-tilda}) imply the first equality
in (\ref{a-optimal-tilda}). For the second equality, we use \eqref{m-h-inverse}, (\ref{m-1}), (\ref%
{m-2})\ and (\ref{v-monotone-tilda}) to obtain%
\begin{equation}
\tilde{\alpha}^{\ast }(x,y,t)=-\frac{p_{y}(x,y,t)}{p_{x}(x,y,t)}%
=h_{y}(p,y,t)=h_{y}(h^{\left( -1\right) }\left( x,y,t\right) ,y,t).
\label{a-h-explicit}
\end{equation}%
The first inequality in (\ref{a-estimate}) follows from (\ref{heat-2}). For
the second, we observe that 
\[
\tilde{\alpha}_{x}^{\ast }(x,y,t)=\frac{h_{yz}(h^{\left( -1\right) }\left(
x,y,t\right) ,y,t)}{h_{z}(h^{\left( -1\right) }\left( x,y,t\right) ,y,t)}, 
\]%
and use (\ref{heat-3}).
\vs
(iii)\ Using the feedback policy (\ref{a-h-explicit}), the state controlled sde
in (\ref{X-Y-tilda}) becomes 
\begin{equation}
d\mathcal{\tilde{X}}_{s}^{\ast }=\tilde{\alpha}^{\ast }(\mathcal{\tilde{X}}%
_{s}^{\ast },\mathcal{Y}_{s},s)dW_{s}^{\mathbb{Q}} \  \text{in} \ (t,T] \ \ 
 \text{and} \ \ \mathcal{\tilde{%
X}}_{t}^{\ast }=x. 
\label{X-tilda-sde}
\end{equation}%
We claim that, for each fixed $\left( y,t\right)\in \R\times[0,T]$,  (\ref{X-tilda-sde})\ has a unique strong solution 
given by 
\begin{equation}\label{X-explicit}
\mathcal{\tilde{X}}_{s}^{\ast }=h\left( h^{\left( -1\right) }\left(
x,y,t\right) ,\mathcal{Y}_{s},s\right) \ \text{in} \ (t,T] \  \text{and} \ \  \mathcal{\tilde{X}}%
_{t}^{\ast }=x. 
\end{equation}%
We begin by  verifying  that $\mathcal{\tilde{X}}^{\ast }$ satisfies (%
\ref{X-tilda-sde}). Indeed, applying Ito's rule to (\ref{X-explicit}) and using  (\ref{heat-h}) gives%
\bee
dh\left( h^{\left( -1\right) }\left( x,y,t\right) ,\mathcal{Y}_{s},s\right)=h_{y}\left( h^{\left( -1\right) }\left(
x,y,t\right) ,\mathcal{Y}_{s},s\right) dW_{s}^{\mathbb{Q}}.
\eee
From (\ref{a-estimate}), we have that, for $s\in [t,T]$, %
\[
\Big( h_{y}\left( h^{\left( -1\right) }\left( x,y,t\right) ,\mathcal{Y}%
_{s},s\right) \Big) ^{2}\leq Ax^{2}+Be^{A\left( T-t\right) }, 
\]%
and, thus, 
\begin{equation*}
\begin{split}
& \mathbb{E}_{\mathbb{Q}}\int_{t}^{T}\left( h_{y}\left( h^{\left( -1\right)
}\left( x,y,t\right) ,\mathcal{Y}_{s},s\right) \right) ^{2} \dd s\\
&\qquad \leq A\big(h^{\left( -1\right) }\left( x,y,t\right)
\big)^{2}(T-t)+\int_{t}^{T}Be^{A\left( T-s)\right) } \dd s.
\end{split}
\end{equation*}%
Moreover,  (\ref{a-h-explicit}) evaluated at (\ref{X-explicit})
yields 
\bee
 \tilde{\alpha}^{\ast }(\mathcal{\tilde{X}}_{s}^{\ast },\mathcal{Y}_{s},s) =
\tilde{\alpha}^{\ast }(h\left( h^{\left( -1\right) }\left( x,y,t\right) ,%
\mathcal{Y}_{s},s\right) ,\mathcal{Y}_{s},s)=h_{y}\left( h^{\left( -1\right) }\left( x,y,t\right) ,\mathcal{Y}%
_{s},s\right), 
\eee
and both equalities in (\ref{optimal-processes-tilda}) follow.
\vs
Lastly, we show that for each fixed  $\left( y,t\right) \in \R\times[0,T],$   (\ref%
{X-tilda-sde})\ has a unique solution. Arguing  by contradiction, we assume 
that there exist two solutions, say $\mathcal{\tilde{X}}_{1}$ and $\mathcal{%
\tilde{X}}_{2},$ with $\mathcal{\tilde X}_{1,t}=\mathcal{\tilde X}_{2,t}=x$. Then,   (\ref{X-tilda-sde}) and  (\ref{a-estimate}) yield, for 
$s\in [t,T],$  that 
\begin{equation}
\begin{split}
\mathbb{E}_{\mathbb{Q}}\left( \mathcal{\tilde{X}}_{1,s}-\mathcal{\tilde{X}}%
_{2,s}\right) ^{2}&=\mathbb{E}_{\mathbb{Q}}\int_{t}^{s}\left( \tilde{\alpha}%
^{\ast }(\mathcal{\tilde{X}}_{1,\rho },\mathcal{Y}_{\rho },\rho )-\tilde{%
\alpha}^{\ast }(\mathcal{\tilde{X}}_{2,\rho },\mathcal{Y}_{\rho },\rho
)\right) ^{2}d\rho \\ \label{uniqueness - estimate}
&\leq \theta _{2}A\mathbb{E}_{\mathbb{Q}}\int_{t}^{s}\left( \mathcal{\tilde{X}%
}_{1,\rho }-\mathcal{\tilde{X}}_{2,\rho }\right) ^{2}d\rho , 
\end{split}
\ee
and uniqueness follows from Gr\"{o}nwall's inequality. 
\end{proof}

\begin{remark}
Alternatively  (\ref{u-tilda-representation})\ may be derived from (%
\ref{u-tilda}) and (\ref{optimal-processes-tilda}), using that at $t=T,$ we
have that $\tilde{\mathcal{X}}_{T}^{\ast }=h\left( h^{\left( -1\right) }\left(
x,y,t\right) ,\mathcal{Y}_{T},T\right) $.
\end{remark}
\vs

We now revert the analysis to the original problem (\ref{u-DFN}), starting
with the specification and regularity of its solution.

\begin{proposition}
The value function $u$ defined in  (\ref{u-DFN}) is the unique in \\$\mathcal{C}^{2,2,1}\left( \mathbb{R\times
R\times }\left[ 0,T\right] \right) $, strictly concave and
strictly increasing in $x,$ solution to  the terminal value HJB\ equation 
\begin{equation}\label{HJB-single}
\begin{cases}
& u_{t}+\max_{\alpha }\left( \frac{1}{2}\alpha ^{2}u_{xx}+\alpha \left(
bu_{x}+u_{xy}\right) \right) +\frac{1}{2}u_{yy}+bu_{y} =\\[1.5mm]
&u_{t}-\dfrac{\left( bu_{x}+u_{xy}\right) ^{2}}{2 u_{xx}}+\frac{%
1}{2}u_{yy}+bu_{y}=0\text{ \ in \ }\mathbb{R\times R\times }\left[
0,T\right), \\[1.75mm]
& u(x,y,T)=J(x),
\end{cases}
\ee
and is given by 
\begin{equation}
u(x,y,t)=\frac{w(x,y,t)}{F(y,t)},  \label{u-w-F}
\end{equation}%
with $w$ as in Proposition 3 and $F$ as in \eqref{F-dfn}. 
\end{proposition}

\begin{proof}
A straightforward adaptation of the results in \cite{Zariphopoulou-MOR}
yield that the value function (\ref{u-DFN}) is the unique viscosity solution
in the class of strictly increasing and strictly concave in $x$ solutions,
for each $\left( y,t\right) \in \mathbb{R\times }\left[ 0,T\right]$, solutions.
\vs

Let $\hat{u}=\dfrac{w}{F}$. We show that $\hat{u}\equiv
u $ on  $\mathbb{R\times R\times }\left[ 0,T\right] .$ 
\vs
For this, it suffices
to establish that $\hat{u}\in \mathcal{C}^{2,2,1}\left( \mathbb{R\times
R\times }\left[ 0,T\right] \right)$ is strictly concave and strictly
increasing in $x$ for each $\left( y,t\right) \in \mathbb{R\times }\left[
0,T\right] ,$ and that it satisfies (\ref{HJB-single}). The first three
properties follow easily from the analogous properties of $w,$ so we only
show that $\hat{u}$ solves (\ref{HJB-single}). 
\vs
Direct
calculations yield,
after substituting the above  in (\ref{HJB-tilda}) and using  that $w$ solves (\ref{HJB-tilda}) and (\ref%
{F-dfn}),  that %
\bee
w_{t}-\frac{1}{2}\frac{\left( w_{xy}\right) ^{2}}{w_{xx}}+\frac{1}{2}w_{yy} =\big( \hat{u}_{t}-\frac{1}{2}\frac{( b\hat{u}_{xy}+\hat{u}%
_{x}) ^{2}}{\hat{u}_{xx}}+\frac{1}{2}\hat{u}_{yy}+b(y,t)\hat{u}%
_{y}\big) F=0,
\eee
Since $F>0$ in $\mathbb{%
R\times }\left[ 0,T\right] $, we easily conclude.

\end{proof}

Next, we introduce two auxiliary functions, $H$ and $k,$ that will be used
in the construction of the optimal processes $X^{\ast }$ and $\alpha ^{\ast
} $. 

\begin{proposition}
Let $k$ solve the terminal value problem \eqref{c-k.2}. 
Then $H:\R\times \R \times [0,T]\to \R$ given by 
\begin{equation}
H(z,y,t)=\left( u_{x}\right) ^{\left( -1\right) }\left(
e^{-z+k(y,t)},y,t\right), 
 \label{h-u-inverse}
\end{equation}%
where $(-1)$ denotes the inverse in the $x-$argument,  is  a $\mathcal{C}%
^{2,2,1}\left( \mathbb{R\times R\times }\left[ 0,T\right] \right) $ 
solution to  the terminal value problem  \eqref{H-intro} and \eqref{takis2} 
where  $c:\R\times[0,T]\to \R$ given by \eqref{c-k.1}
solves the terminal value problem
\begin{equation}\label{c-heat-eqn}
c_{t}+\frac{1}{2}c_{yy}=0\text{ \ in \ }\mathbb{R\times }\left[ 0,T\right)\ \ \text{and} \ \ 
c(y,T)=b(y,T).
\end{equation}%
\end{proposition}

Looking at \eqref{w-h-ez}  and \eqref{h-u-inverse}, it is clear that $H$  resembles the auxiliary function $h$.
We stress, however, that contrary to \eqref{w-h-ez}, the equation satisfied by $H$ is not uniformly
elliptic, and, therefore, it is not a priori known whether a smooth solution
exists.

\begin{proof}
It follows easily from the assumptions that $k$ is well defined, since for all $(y,t)\in \R\times [0,T]$, %
\[
-\theta _{2}(T-t)<k(y,t)<-\theta _{1}(T-t). 
\]%
Furthermore, (\ref{u-w-F}) yields that $u_{x}>0$ in $\mathbb{R\times
R\times }\left[ 0,T\right] $\ and, therefore,  for each $\left( y,t\right) \in \mathbb{R\times }\left[ 0,T\right],$  the map  
 $x\to u_{x}(x,y,t)$ is invertible, and, thus, $H\ $is well defined. 
 \vs
 Rewriting (\ref{h-u-inverse}) as 
\begin{equation}
u_{x}(H(z,y,t)y,t)=e^{-z+k(y,t)},  \label{u-aux}
\end{equation}%
and using (\ref{w-h-ez}) and (\ref{u-w-F}), we obtain that, for each $(z,y,t) \in \mathbb{R\times }\R \times \left[ 0,T\right]$, 
\begin{equation}
H(z,y,t)=h(z-n(y,t),y,t), 
\label{h-h(0)}
\end{equation}%
with 
\begin{equation}
n=k+\ln F.  \label{n(y,t)-dfn}
\end{equation}%
Then, (\ref{h-h(0)})\ yields 
\[
H_{t}=-n_{t}h_{z}+h_{t},\text{ }H_{z}=h_{z},\text{ }H_{zz}=h_{z},  \ H_{xy}=-n_{y}h_{zz}+h_{zy},\text{ }H_{y}=-n_{y}h_{z}+h_{y},
\]%
and%
\[
H_{yy}=-n_{yy}h_{z}+n_{y}^{2}h_{zz}-2n_{y}h_{zy}+h_{yy}, 
\]%
where $h$ and its derivatives are evaluated at $\left( z-n(y,t),y,t\right) .$%
\vs
Moreover, (\ref{n(y,t)-dfn})\ gives 
\begin{equation*}
\frac{F_{y}}{F}+k_{y}-n_{y}=0\text{ \  \ in \ }\mathbb{R\times }\left[ 0,T%
\right].  
\end{equation*}%

Using (\ref{F-dfn}), (\ref{F-equation}) and (\ref{h-u-inverse}), we
deduce that 
\bee
n_{t}+\frac{1}{2}n_{yy}=\frac{1}{F}\big( F_{t}+\frac{1}{2}F_{yy}\big) +\big( k_{t}+\frac{1}{2}%
k_{yy}-\left( \frac{F_{y}}{F}\right) ^{2}\big)=k_{t}+\frac{1}{2}k_{yy}-\frac{1}{2}b^{2}=0.\\ 
\eee

Recalling \eqref{c-k.1} we obtain, after some routine calculations,  that
\bee
\begin{split}
&H_{t}+\frac{1}{2}c^{2}H_{zz}+cH_{zy}+\frac{1}{2}H_{yy} \\
&=\big( h_{t}+\frac{1}{2}%
h_{yy}\big) -\big( n_{t}+\frac{1}{2}n_{yy}\big) h_{z}\\
&+\frac{1}{2}\big( \frac{F_{y}}{F}+k_{y}-n_{y}\big) ^{2}h_{zz}+\big( 
\frac{F_{y}}{F}+k_{y}-n_{y}\big) h_{zy}=0. 
\end{split}
\eee

Finally, to show (\ref{c-heat-eqn}) we use \eqref{c-k.1} and \eqref{c-k.2} to obtain 
\[
c_{t}+\frac{1}{2}c_{yy}=b_{t}+\frac{1}{2}b_{yy}+bb_{y}. 
\]%
On the other hand, (\ref{F-dfn})  yields%
\[
F_{yt}=b_{t}F+bF_{t}\text{ \ and \ }F_{yyy}=b_{yy}F+2b_{y}F_{y}+bF_{yy}, 
\]%
which together with (\ref{F-equation}) imply 
\bee
\begin{split}
b_{t}F+bF_{t}+\frac{1}{2}\left( b_{yy}F+2b_{y}F_{y}+bF_{yy}\right)= \big( b_{t}+\frac{1}{2}b_{yy}+b_{y}\frac{F_{y}}{F}\big) F=0. \\
\end{split}
\eee
The claim follows  since $F>0$ in $\mathbb{R\times }\left[ 0,T\right] .$

\end{proof}

We are now ready to state  the main results which provide closed-form
expressions for the value function $u$, the optimal wealth
process  $\mathcal{X}^\star$ and the associated optimal control process $\alpha^\star$.
As mentioned earlier, these results are, to the best of
our knowledge, new and of independent interest.

\begin{theorem}\label{ord200}
The value function $u$ in (\ref{u-DFN}) is represented, for $(x,y,t)\in \R\times \R\times [0,T]$,   %
 as%
\begin{equation} \label{u-thm}
u(x,y,t)= \mathbb{E}\left[ \left. J\left( H(H^{(-1)}(x,y,t)+L_{t,T},Y_{T},T)\right)
\right \vert X_{t}=x,Y_{t}=y\right], 
\end{equation}%
with the process $L$ defined by \eqref{L-intro}
with $H$ and $c$ solving  respectively   \eqref{H-intro}, \eqref{takis2}  and \eqref{c-heat-eqn}. 
\vs
The optimal wealth process $\mathcal{X}^{\ast }$ is as 
\begin{equation} \label{x-thm}
\mathcal{X}_{s}^{\ast }=H(H^{(-1)}(x,y,t)+L_{t,s},Y_{s},s) \ \text{in} \  (t,T] \ \ \text{and} \ \ \mathcal{X}_{t}^{\ast }=x,
\end{equation}%
and 
the optimal feedback control $\alpha ^{\ast }$ is given by%
\begin{equation}
\alpha ^{\ast }=-\frac{bu_{x} +u_{xy}}{u_{xx} }\text{ \  \ in \ }\mathbb{R\times
R\times }\left[ 0,T\right] ,  \label{a-R-thm}
\end{equation}%
and satisfies, with $\tilde{\alpha}^{\ast }$ as in (\ref{a-optimal-tilda}), %
\begin{equation}
\alpha ^{\ast }=\tilde{\alpha}^{\ast } \text{ \ in \ }\mathbb{%
R\times R\times }\left[ 0,T\right].  \label{pi-equals-alpha}
\end{equation}%
\vs
Moreover, if  $A$ and $B$ are as in \eqref{utility}, then, for each $(x,y,t)\in \mathbb{%
R\times R\times }\left[ 0,T\right]$,  
\begin{equation}
0<\alpha ^{\ast }(x,y,t)\leq \theta _{2}\sqrt{Ax^{2}+Be^{A\left( T-t\right) }%
}\text{\  and \  }\alpha _{x}^{\ast }(x,y,t)\leq \theta _{2}A .   \label{a-estimates}
\end{equation}%

The optimal control process $\alpha ^{\ast }$ is given by 
\begin{equation}
\alpha _{s}^{\ast }=\alpha ^{\ast }(\mathcal{X}_{s}^{\ast },Y_{s},s),
\label{a-feedback-process}
\end{equation}%
with $\mathcal{X}^{\ast }$ solving the sde
\begin{equation}
d\mathcal{X}_{s}^{\ast }=b(Y_{s},s)\alpha ^{\ast }(\mathcal{X}_{s}^{\ast
},Y_{s},s)ds+\alpha ^{\ast }(\mathcal{X}_{s}^{\ast },Y_{s},s)dW_{s} \ \;\text{in} \ \; (t,T] \ \text{ \ and  
\ } \ \mathcal{X}_{t}^{\ast }=x,  \label{x-feedback}
\end{equation}%
and is represented, for $s\in [t,T]$, as 
\begin{equation}\label{a-process-thm}
\begin{split}
\alpha _{s}^{\ast }&=c(Y_{s},s)H_{z}(H^{(-1)}(x,y,t)+L_{t,s},Y_{s},s)\\
& \qquad \qquad +H_{y}(H^{(-1)}(x,y,t)+L_{t,s},Y_{s},s)\\
&=\alpha ^{\ast }\left( H\left( H^{\left( -1\right) }\left( x,y,t\right)
+L_{t,s},Y_{s},s\right) ,Y_{s},s\right).
\end{split}
\ee
\end{theorem}

\begin{proof}
We only establish the claims for $\mathcal{X}^\ast$ and $\alpha^\ast$,  since (\ref{u-thm}) follows easily.
\vs

The first order conditions in (\ref{HJB-single}) yield \eqref{a-R-thm}, 
while  (\ref{u-w-F})\ gives 
\[
-\frac{bu_{x} +u_{xy}}{u_{xx} }=-\frac{w_{xy}}{w_{xx} }, 
\]%
and (\ref{pi-equals-alpha})\ follows, while  (\ref{a-estimates}) comes 
directly from (\ref{a-estimate}).
\vs

Next, we note that (\ref{a-feedback-process}) and (\ref{x-feedback}) will
follow once we show that the candidate control process is admissible and
that   (\ref{x-feedback}) has a unique strong solution. 
\vs
To this end, observe
that (\ref{u-aux}) gives 
\[
H_{z}u_{xx}(H,y,t)=-u_{x}(H,y,t), 
\]%
and%
\[
H_{y}u_{xx}(H,y,t)+u_{xy}(H,y,t)=k_{y}\left( y,t\right) u_{x}(H,y,t), 
\]%
and, in turn, \eqref{c-k.1} and \eqref{a-R-thm} imply%
\begin{equation}
\alpha ^{\ast }(H,y,t)=c(y,t)H_{z}(z,y,t)+H_{y}(z,y,t).
\label{a-feedback-reduced}
\end{equation}%
Using (\ref{a-feedback-reduced}), \eqref{L-intro} and \eqref{Y-reduced}, we write (\ref{x-feedback}) as%
\be\label{optimal sde}
\begin{split}
&d\mathcal{X}_{s}^{\ast } =H_{z}\left( H^{(-1)}(\mathcal{X}_{s}^{\ast },Y_{s},s),Y_{s},s\right) \left(
b(Y_{s},s)c(Y_{s},s)ds+c(Y_{s},s)dW_{s}\right)\\ 
&\qquad +H_{y}\left( H^{(-1)}(\mathcal{X}_{s}^{\ast },Y_{s},s),Y_{s},s\right) \left(
b(Y_{s},s)ds+dW_{s}\right)\\
&=H_z(H^{(-1)}({\mathcal{X}}^\ast_s, \mathcal{Y}_s,s),\mathcal Y_s,s) \dd L_{t,s} 
+H_y(H^{(-1)}({\mathcal{X}}^\ast_s, \mathcal{Y}_s,s),\mathcal Y_s,s)\dd \mathcal{Y}_s.
\end{split}
\end{equation}%
To identify the unique, for each fixed $\left( y,t\right) $ in $%
\mathbb{R\times }\left[ 0,T\right] ,$ solution to (\ref{optimal sde}),  we introduce the process $\mathcal{\hat{X}}$ 
given by 
\[
\mathcal{\hat{X}}_{s}=H(H^{(-1)}(x,y,t)+L_{t,s},Y_{s},s) \ \text{ for } \ s\in [t,T].
\]%
Using \eqref{H-intro} we find
%
\bee
\begin{split}
&d\mathcal{\hat{X}}_{s}=d\left( H(H^{(-1)}(x,y,t)+L_{t,s},Y_{s},s)\right) \\
&=\left( H_{z}\left( H^{(-1)}(\mathcal{\hat{X}}_{s},Y_{s},s),Y_{s},s\right)
\right) dL_{t,s}  +H_{y}\left( H^{(-1)}(\mathcal{\hat{X}}_{s},Y_{s},s),Y_{s},s%
\right) dY_{s}\\ 
&+ \left(H_{t}(H^{(-1)}(x,y,t)+L_{t,s},Y_{s},s) 
+\frac{1}{2}%
c^{2}(Y_{s},s)H_{zz}(H^{(-1)}(x,y,t)+L_{t,s},Y_{s},s)\right) \dd s\\
&+\left(c(Y_{s},s)H_{zy}(H^{(-1)}(x,y,t)+L_{t,s},Y_{s},s) 
+\frac{1}{2}
H_{yy}(H^{(-1)}(x,y,t)+L_{t,s},Y_{s},s) \right)\dd s\\
&= H_{z}\left( H^{(-1)}(\mathcal{\hat{X}}_{s},Y_{s},s),Y_{s},s\right)
\dd L_{t,s} +H_{y}\left( H^{(-1)}(\mathcal{\hat{X}}_{s},Y_{s},s),Y_{s},s%
\right) dY_{s},
\end{split}
\eee
and, thus,
\bee
d \tilde{\mathcal{X}}^\ast_s=d  \hat{\mathcal{X}}^\ast_s \ \text{in} \ (t,T] \ \ \text{and} \ \ \tilde{\mathcal{X}}^\ast_t=\hat{\mathcal{X}}^\ast_t=x. 
\eee

The first claim of the theorem is a direct consequence of (\ref{x-thm}) evaluated at $T$ and (\ref%
{u-DFN}), while the last claim  follows from (\ref{x-thm}), (\ref{a-feedback-process}) and (\ref%
{a-feedback-reduced}).
\vs
It remains to show that the It\^{o}'s integrals appearing in the expansions above are well defined, and that
\eqref{optimal sde} has a unique solution for each $(y,t)\in \R\times [0,T]$. 
\vs
These follow from \eqref{pi-equals-alpha} and arguments similar, albeit more tedious, to the ones used in the proof of 
Proposition~2.

\end{proof}
\begin{remark}
\bigskip We note that   (\ref{w-h-ez}) and (\ref%
{h-u-inverse}),  rewritten below as 
\begin{equation}
w_{x}(h(z,y,t), y,t)=e^{-z}\text{ \ and \ }u_{x}(H(z,y,t)y,t)=e^{-z+k(y,t)}%
\label{transformations}
\end{equation}%
for problems (\ref{u-tilda}) and (\ref{u-DFN}), relate   the 
marginals $w_{x}$ and $u_{x}$ to their corresponding linear equations (\ref%
{heat-h}) and (\ref{H-intro}). In the existing literature of
portfolio choice in complete markets, linearization has been carried out
directly for the inverse marginal $\left( u_{x}\right) ^{\left( -1\right) }$
leading  to a single linear pde. Herein, however, we use a different
linearization approach, by seeking a pair of auxiliary functions $(H,k)$
solving respective linear equations. The pair $\left( h,0\right) $ is easily
identified as the analogue of $(H,k)$ when $b(y,t)\equiv 0$ and the modified utility $J(x)F(y,T)$. Of course, the
two approaches are equivalent but working with transformations (\ref%
{transformations}) allow us to  produce regularity results and produce
closed-form solutions for general utilities. 
\end{remark}

We conclude with the derivation of an \text{autonomous} equation satisfied by the
optimal feedback controls for problems (\ref{u-tilda}) and (\ref{u-DFN}). 
The equation generalizes  the  fast-diffusion equation that the risk tolerance
function in the log-normal case satisfies; see, among others, \cite{Black}, \cite{Huang-Z.}, \cite{kallblad-Z} and \cite{Monin-Zariphopoulou}.

\begin{proposition}
Let $r$ be the risk tolerance (\ref{risk}). The optimal feedback control functions $\tilde{a}^{\ast }$ and $%
\alpha ^{\ast }$  for problems (\ref{u-tilda}) and (\ref{u-DFN})\
satisfy the terminal value problem%
\begin{equation}
\begin{cases}
R_{t}+\frac{1}{2}R^{2}R_{xx}+RR_{xy}+\frac{1}{2}R_{yy}=0\  \text{\  \ }in\text{
\ }\mathbb{R\times R\times }\left[ 0,T\right),\\[1.5mm]  \label{R-equation}
R(x,y,T)=b(y,T)r(x).
\end{cases}
\end{equation}%
\end{proposition}

\begin{proof}
Given   (\ref{pi-equals-alpha}), it suffices to establish (\ref%
{R-equation}) for $\tilde{a}^{\ast }$. 
\vs
Using \eqref{a-optimal-tilda} and differentiating 
$p=h^{\left( -1\right) }$ we find   
\[
\tilde{\alpha}_{t}^{\ast }=p_{t}h_{zy}+h_{yt},\text{ \  \ }\tilde{\alpha}%
_{x}^{\ast }=p_{x}h_{zy}, 
\]%
\[
\tilde{\alpha}_{xx}^{\ast }=p_{xx}h_{zy}+\left( p_{x}\right) ^{2}h_{yzz},%
\text{ \ }\tilde{\alpha}_{y}^{\ast }=p_{y}h_{zy}+h_{yy}, 
\]%
\[
\tilde{\alpha}_{xy}^{\ast }=p_{xy}h_{zy}+p_{x}p_{y}h_{zzy}+p_{x}h_{zyy}, 
\]%
and 
\[
\tilde{\alpha}_{yy}^{\ast }=p_{yy}h_{zy}+\left( p_{y}\right)
^{2}h_{zzy}+2p_{y}h_{zyy}+h_{yyy}, 
\]%
where functions $\tilde{\alpha}_{x}^{\ast }, p$ and their derivatives are
evaluated at $\left( x,y,t\right) ,$ and $h_{y}$ and its derivatives at $%
\left( p(x,y,t\right) ,y,t)$. 
\vs
Grouping terms yields%
\bee
\begin{split}
& \tilde{\alpha}_{t}^{\ast }+\frac{1}{2}\left( \tilde{\alpha}^{\ast }\right)
^{2}\tilde{\alpha}_{xx}^{\ast }+\tilde{\alpha}^{\ast }\tilde{\alpha}%
_{xy}^{\ast }+\frac{1}{2}\tilde{\alpha}_{yy}^{\ast } \\
&=\big( p_{t}+\frac{1}{2}\left( h_{y}\right) ^{2}p_{xx}+h_{y}p_{xy}+\frac{1}{%
2}p_{yy}\big) h_{zy} 
+\frac{1}{2}\left( h_{y}p_{x}+p_{y}\right) ^{2}h_{zzy}\\
&=\big( p_{t}+\frac{1}{2}\left( h_{y}\right) ^{2}p_{xx} +h_{y}p_{xy}+\frac{1}{%
2}p_{yy}\big) h_{zy}, 
\end{split}
\eee
where we used that $h_{y}$ satisfies  \eqref{heat-h}
  and that $%
h_{y}p_{x}+p_{y}=0$ (cf. (\ref{m-1}) and (\ref{m-2})). 
\vs
Using (\ref{m-2}), \eqref{m-3} and (%
\ref{m-4}), we obtain%
\[
p_{t}+\frac{1}{2}\left( h_{y}\right) ^{2}p_{xx}+h_{y}p_{xy}+\frac{1}{2}%
p_{yy}=0, 
\]%
and we conclude.
\end{proof}

\subsection{Examples}

We discuss three representative examples of utilities. We start with the
well studied exponential one, which we review for completeness and also to
provide the formulae to be used in the upcoming MFG\ examples. Then, we turn 
to  SAHARA\ utilities and, finally, to  their
CMIM\ extension. 

\subsection{Exponential utility}


To ease the presentation, since the arguments are similar, we only consider the utiity  $J:\R\to \R$ given by 
\begin{equation}
J(x)=-e^{-x}.  \label{J-expo}
\end{equation}%
\vs
Since, for $z>0$,  $I(z)=-\ln z$, (\ref{takis2}) yields $%
H(z,y,T)=z$. Then, (\ref{H-intro}) implies that $H(z,y,t)=z$ for each $(z,y,t)\in \mathbb{R\times
R\times }\left[ 0,T\right] $, and, in turn, (\ref{x-thm}) and (\ref%
{a-process-thm})\ give
\begin{equation}
\mathcal{X}_{s}^{\ast }=x+L_{t,s}\text{ \  \ and \  \ }\alpha _{s}^{\ast
}=c(Y_{s},s),  \label{X-a-optimal-expo-Merton}
\end{equation}%
with $L$ as in \eqref{L-intro} and $c$ solving (\ref{c-heat-eqn}).
\vs
Furthermore, for $(x,y,t) \in  \mathbb{R\times R\times }\left[ 0,T\right]$ and $k$ as in (\ref{c-k.2}),
\begin{equation*}
u(x,y,t)=\mathbb{E}\left[ -\exp \left( -\left( x+L_{t,T}\right) \right) %
\right] =-e^{-x+k(y,t)}.   
\end{equation*}%

\subsection{Utilities with asymptotically linear risk tolerance}

We consider utility functions $J$ characterized by their risk tolerance
function of the parametric form (\ref{double}).
Such functions were  introduced
in \cite{Musiela-Za-monotone} and the case of known drift,  $b(y,t)\equiv b,$ 
 was studied in \cite{zariphopoulou-zhou} and \cite{Chen-SAHARA}, where 
 the acronym SAHARA\ (symmetric asymptotic hyperbolic risk
aversion)\ utilities was introduced. 
\vs
To ease the presentation, we work only with the special case that the risk tolerance is 
\[
r(x)=\sqrt{x^{2}+1}, 
\]
%
which, from \eqref{takis2} and (\ref{risk}), corresponds, for $(z,y) \in \R\times \R$,  to 
\begin{equation}\label{H-sinh}
H(z,y,T)=\sinh z.  
\end{equation}%
It follows from  \eqref{H-intro}
that, for $(z,y,t) \in \R\times \R \times [0,T]$,
\begin{equation}\label{H-SAHARA}
H(y,t)=\frac{1}{2}e^{z}l_{1}(y,t)-\frac{1}{2}e^{-z}l_{2}(y,t), 
\end{equation}%
where   $l_{1}, l_2:\R\times [0,T]\to \R$ are smooth solutions to 
\[
l_{1,t}+\frac{1}{2}l_{1,yy}+cl_{1,y}+\frac{1}{2}c^{2}l_{1}=0\text{
 in }\mathbb{R\times }\left[ 0,T\right) \ \text{and} \ l_{1}(y,T)=1, 
\]%
and 
\[
l_{2,t}+\frac{1}{2}l_{2,yy}-cl_{2,y}+\frac{1}{2}c^{2}l_{2}=0\text{
 in  }\mathbb{R\times }\left[ 0,T\right)  \ \text{and} \ l_{2}(y,T)=1.
\]%
Since, for $i=1,2,$ $l_{i}>0$ in $%
\mathbb{R\times }\left[ 0,T\right]$, for $(z,y,t)\in \R\times \R \times [0,T]$, we have 
\begin{equation}
H_{z}(z,y,t)=\frac{1}{2}e^{z}l_{1}(y,t)+\frac{1}{2}e^{-z}l_{2}(y,t)>0,  \label{H'-SAHARA}
\end{equation}%
and, thus, $H^{(-1)}$, which is inverse in the $z-$argument,  exists and is given, for $(x,y,t)\in \mathbb{R\times R\times }\left[ 0,T\right]$,   by 
\begin{equation}
H^{\left( -1\right) }(x,y,t)=\ln \frac{2x+\sqrt{4x^{2}+2l_{2}(y,t)}}{%
2l_{1}(y,t)}. 
\label{H-inverse-SAHARA}
\end{equation}%
Following (\ref{x-thm}), (\ref{a-process-thm}), (\ref{H-SAHARA}), (%
\ref{H'-SAHARA}) and (\ref{H-inverse-SAHARA}), the optimal processes $\mathcal{X}^{\ast }$ and $\alpha ^{\ast }$ are given by 

\[
\mathcal{X}_{s}^{\ast }=\frac{1}{2}e^{H^{\left( -1\right) }\left(
x,y,t\right) +L_{t,s}}l_{1}(Y_{s},s)-\frac{1}{2}e^{-\left( H^{\left(
-1\right) }\left( x,y,t\right) +L_{t,s}\right) }l_{2}(Y_{s},s), 
\]%
and 
\[
\mathcal{\alpha }_{s}^{\ast }=\frac{1}{2}e^{H^{\left( -1\right) }\left(
x,y,t\right) +L_{t,s}}l_{1}(Y_{s},s)+\frac{1}{2}e^{-\left( H^{\left(
-1\right) }\left( x,y,t\right) +L_{t,s}\right) }l_{2}(Y_{s},s). 
\]

\subsection{Utilities with completely monotonic inverse marginals}

A large class of utilities, which substantially extend the SAHARA ones, 
are utilities with completely monotonic (CM) inverse marginals (CMIM).
\vs
Classical results for the representation of CM\ functions yield that the
inverse marginals of CMIM utilities  are  of the form%
\begin{equation}
I(x)=\int x^{-\rho }d\mu \left( \rho \right) ,
\label{CMIM}
\end{equation}%
for a suitable measure so that $I^{\prime }(x)<0,$ $\lim_{x\rightarrow
0}I(x)=\infty $ and $\lim_{x\rightarrow \infty }I(x)=0$.
\vs
CMIM\ utilities were introduced in \cite{Musiela-Za-monotone} in the context
of time-monotone forward utilities where a detailed study of the measure $%
\mu $ was carried out; see, also, \cite{Geng-Zariphopoulou} for related
turnpike-type problems and \cite{Mostovi-et-al} for their role in abstract
semi-martingale market models.
\vs
From (\ref{CMIM}) and (\ref{u-aux}), we deduce that CMIM\ utilities can be
equivalently described by the related function $H(z,y,T)$, which takes, for $z\in \R$,  the
form 
\begin{equation}
H(z,y,T)=\int e^{z\rho }\dd \mu \left( \rho \right).  \label{H-CMIM-terminal}
\end{equation}%
For example, (\ref{H-sinh}) corresponds to measure $\mu \left( \rho
\right) =\frac{1}{2}\left( \delta _{1}-\delta _{-1}\right) $.
\vs
With terminal condition (\ref{H-CMIM-terminal}), the solution to \eqref{H-intro}
is represented as%
\begin{equation}
H(z,y,t)=\int e^{z\rho }m_{\rho }(y,t)d\mu \left( \rho \right) ,
\label{H-CMIM}
\end{equation}%
where for each $\rho \in  \text{supp} \left( \mu \right)$,  $l_{\rho }:\R\times[0,T]\to \R^+$ 
solves the terminal value problem
\begin{equation*}
l_{\rho ,t}+\frac{1}{2}l_{\rho ,yy}+(\rho c) l_{\rho ,y}+\frac{1}{2}(\rho c)^{2}l_{\rho }=0\text{ in  }\mathbb{R\times }\left[ 0,T\right) \ \text{and} \  l_{\rho }(y,T)=1.
\end{equation*}%

The optimal processes $\mathcal{X}%
^{\ast }$ and $\alpha ^{\ast }$ are given, for $0\leq t\leq s\leq T$, by 
\begin{equation*}
\mathcal{X}_{s}^{\ast }=\int e^{\rho \left( H^{\left( -1\right)
}(x,y,t)+L_{s,t}\right) }l_{\rho }(Y_{s},s)d\mu \left( \rho \right) ,  
\end{equation*}%
and 
\begin{equation*}
\mathcal{\alpha }_{s}^{\ast }=\int \rho e^{\rho \left( H^{\left( -1\right)
}(x,y,t)+L_{s,t}\right) }l_{\rho }(Y_{s},s)d\mu \left( \rho \right).  
\end{equation*}%
Naturally, one needs to specify proper conditions on both the measure $\mu $
and the market dynamics so that the above quantities are well-defined. We
leave these questions for future research.

\begin{remark}
The form of the auxiliary function $H$ in (\ref{H-sinh}) and (\ref%
{H-CMIM-terminal}) allows us to guess solutions of the form (\ref{H-SAHARA}) and
(\ref{H-CMIM}), respectively. These candidate solutions are smooth as they
solve individual one-dimensional linear pdes. This together with the
uniqueness of viscosity solutions for equation (\ref{H-intro})\ yield the
result. However, outside the class of CMIM\ utilities, smooth solutions are
not a priori guaranteed as the problem is two-dimensional and not uniformly
elliptic.
\end{remark}

\section{The N-player and the mean field game in the reduced complete
information market}

We consider a game of $N-$players in the market as in section 2, consisting
of a stock with price $S$ solving (\ref{stock-reduced}) with the local
factor process $Y$ satisfying (\ref{Y-reduced}), and a bond offering zero
interest rate. The latter is taken to be the numeraire and all state and
controlled processes below are expressed in discounted by it units.
\vs
The wealth process $X^{N,i}$ of the $i^{th}-$player, for $i=1,...,N,$
evolves according to the sde%
\begin{equation}
dX_{s}^{N,i}=b(Y_{s},s)\pi _{s}^{N,i}ds+\pi _{s}^{N,i}dW_{s} \ \text{in \ } (t,T] \ \text{and \ }  
X_{t}^{N,i}=x_{i}\in \R,
\label{wealth-single}
\end{equation}%
with the drift coefficient $b$ as in (\ref{F-dfn}) and the Brownian motion 
$W$ as in (\ref{W-innovations}). 
\vs
The control policies $\pi ^{N,i}$, %
which are self-financing and satisfy $\pi ^{N,i}\in \mathcal{A}$ as in
(\ref{admissible-set}), model the (discounted)\ amount invested in the
stock.
\vs
To ease the presentation, in what follows we will be occasionally using the abbreviated notation $%
z^{1:N}=(z_{1},...,z_{N})$. Moreover, we will not be repeating, unless important, the fact that $i=1,
\ldots, N$. 
\vs
The value function of the  $i^{th}-$player, for $(x_{1},...,x_{N},y,t) \in \mathbb{%
R^N\times R\times }\left[ 0,T\right]$, is 
\be\ \label{V-i}
\begin{split}
& u^{i}\left( x_{1},...,x_{N},y,t\right) \\
&=\sup_{\pi _{i}\in \mathcal{A}}\mathbb{E}\Big[ J\Big( X_{T}^{N,i},%
\frac{1}{N-1}\sum \limits_{j=1,j\neq i}^{N}\delta _{X_{T}^{N,j}}\Big)
 \vert X_{t}^{N,1:N}=x_{1:N},Y_{t}=y\Big]. 
\end{split}
\end{equation}%
The utility function denoted, by a slight abuse of notation, as $J$ is
common across players and depends on both their individual terminal wealth $%
X_{T}^{N,i}$ and the law of their peers' wealth at $T$. As mentioned above,
since all involved processes are unitless, there is no issue with unit
consistency in nonlinear payoffs as above, thus allowing for couplings that depend
on higher moments, quantiles, distances among returns, and others; see \cite{Musiela-Z-numeraire} for the role of unit consistency,
\vs
We recall that a control process $\left( \pi ^{N,1,\ast },...,\pi ^{N,N,\ast
}\right) $ is a \text{Nash equilibrium }of this game, if, for each $%
i=1,...,N,$ and all $\pi ^{N,i}\in \mathcal{A}$,%
\bee
\begin{split}
&\mathbb{E}\Big[ J\Big( X_{T}^{N,i},\frac{1}{N-1}\sum
\limits_{j=1,j\neq i}^{N}\delta _{X_{T}^{N,j,\ast }}\Big)  \vert %
 X_{t}^{N,1:j-1,\ast
}=x_{1:i-1},X_{t}^{N,i}=x_{i},\\
&\hskip1in X_{t}^{N,j+1:N,\ast }=x_{i+1:N},Y_{t}=y
\big]\\
&\leq \mathbb{E}\Big[  J\Big( X_{T}^{N,i,\ast },\frac{1}{N-1}\sum
\limits_{j=1,j\neq i}^{N}\delta _{X_{T}^{N,j,\ast }}\Big)  \vert
X_{t}^{N,1:N,\ast }=x_{1:N},Y_{t}=y\Big ],
\end{split}
\eee
where  ${X^{N,j,\ast }}$ denotes the solution to \eqref{wealth-single}
with control process $\pi^{N,j,\ast}$. 
\vs
Next, we (formally) assume that there exist Nash
equilibrium control processes $\left( \pi ^{N,i,\ast },...,\pi ^{N,N,\ast
}\right) $ in the feedback form 
\[
\pi _{s}^{N,i,\ast }=\pi ^{N,i,\ast }\left( X_{s}^{N,1,\ast
},...,X_{s}^{N,N,\ast },Y_{s},s\right) ,\text{ } 
\]%
for some functions $\pi ^{N,i,\ast }:\mathbb{R}^{N}\mathbb{\times R\times }%
\left[ 0,T\right] \rightarrow \mathbb{R}$ and processes $X^{N,i,\ast }$
solving (\ref{wealth-single}) with controls $\pi ^{N,i,\ast }$.
Then, the value functions $\big(u^{N,1},\ldots, u^{N,N}\big) $ are expected to satisfy 
 the fully coupled system of terminal value Hamilton-Jacobi-Bellman (HJB)\ equations 
\begin{equation}\label{HJB-finite}
\begin{split}
&u_{t}^{N,i}+\max_{\pi ^{N,i}}\Big(\frac{1}{2}\left( \pi ^{N,i}\right)
^{2}u_{x_{i}x_{i}}^{N,i}\\
&\qquad +\pi ^{N,i}\Big( b
u_{x_{i}}^{N,i}+\sum \limits_{j=1,\text{ }j\neq i}^{N}\pi ^{N,j,\ast
}u_{x_{i}x_{j}}^{N,i} \big)+ \pi ^{N,i} u_{x_{i}y}^{N,i} \Big)\\
&\qquad+\frac{1}{2}\sum \limits_{j=1,\text{ }j\neq i}^{N}\sum \limits_{k=1,k\neq
i}^{N}\pi ^{N,j,\ast }\pi ^{N,k,\ast }u_{x_{j}x_{k}}^{N,i}+\sum
\limits_{j=1,\text{ }j\neq i}^{N}\pi ^{N,j,\ast }u_{x_{j}y}^{N,i}\\
&\qquad+b \sum \limits_{j=1,\text{ }j\neq i}^{N}\pi ^{N,j,\ast
}u_{x_{j}}^{N,i} 
+\frac{1}{2}u_{yy}^{N,i} +b  u_{y}^{N,i}=0\text{ \  \ in \ }%
\mathbb{R}^{N}\times \mathbb{R}\times \left[ 0,T\right],\\
&u^{N,i}\Big( x_{1},...,x_{N},y,T\Big) =J\Big( x_{i},\dfrac{1}{N-1}\sum
\limits_{j=1,j\neq i}^{N}\delta _{x_{j}}\Big).
\end{split}
\ee
If, in addition, the maximum in (\ref{HJB-finite}) is well-defined in each
respective HJB\ equation, we deduce - still formally - that the optimal
feedback control functions $\left( \pi ^{N,1,\ast },...,\pi ^{N,N,\ast
}\right) $ must satisfy the linear $N\times N$ system 
\begin{equation}
\pi ^{N,i,\ast }u_{x_{i}x_{i}}^{N,i}+\sum \limits_{j=1,\text{ }j\neq
i}^{N}\pi ^{N,j,\ast }u_{x_{i}x_{j}}^{N,i}=-b
u_{x_{i}}^{N,i}-u_{x_{i}y}^{N,i},\text{\  \ }i=1,...,N.  \label{optN}
\end{equation}

This system, however, is not tractable due to the interlinked dependence of
the feedback controls $\pi ^{N,i,\ast \prime }s$ and the coefficients $%
u_{x_{i}}^{N,i},u_{x_{i}x_{i}}^{N,i},u_{x_{i}x_{j}}^{N,i}$ and $%
u_{x_{i}y}^{N,i},$ $i,j=1,...,N$. We actually note that it is not even clear
if the individual value functions $u^{N,i\prime }s$ are smooth enough for
the latter derivatives to exist. Similar difficulties were, also,
encountered in the simpler system derived in the Black and Scholes market
model ($b(y,t)\equiv b$) in \cite{Souganidis-Z-MFG}.
\vs
Motivated by the intractability of the finite player game, we consider next
a related mean field game.

\subsection{The mean field game and its master system}

The representative agent's state $X$ solves, for 
$\pi \in \mathcal{A},$ the continuum analogue of sde \eqref{wealth-single}, that is,
\begin{equation}
dX_{s}=b(Y_{s},s)\pi _{s}ds+\pi _{s}dW_{s} \ \text{in} \ (t,T] \text{ \ and  \ }X_{t}=x,  \label{X-representative}
\end{equation}%
with $b$ as in \eqref{F-dfn}, and $W$ and $Y$ as in (\ref{W-innovations})
and (\ref{Y-reduced}), respectively.
\vs
We emphasize that the derivation below is formal since we do not have
estimates that would allow us to give a rigorous proof. To this end, for
large number of players, we assume that the optimal feedback controls $\pi
^{N,i,\ast }\left( x_{1},...,x_{N},y,t\right) $ and the value functions $%
u^{N,i}\left( x_{1},...,x_{N},y,t\right) $ satisfy, for each $i=1,...,N,$
\begin{equation}\label{paris1}
\pi ^{N,i,\ast }\left( x_{1},...,x_{N},y,t\right) \text{ }\simeq \text{ }\pi
^{N,\ast }\left( x_{i},y,m^{N,i},t\right) , 
\end{equation}
and 
\begin{equation}\label{paris2}
u^{N,i}\left( x_{1},...,x_{N},y,t\right) \simeq u^{N}\left(
x_{i},y,m^{N,i},t\right) , 
\end{equation}%
where $m^{N,i}$ is the empirical measure of the rest
of the players, 
\begin{equation}\label{paris3}
m^{N,i}=\dfrac{1}{N-1}\sum \limits_{j=1,j\neq i}^{N}\delta _{x_{j}}. 
\end{equation}%
We assume that, for each $i,$ as $N\rightarrow \infty ,$ $%
m^{N,i}$ converge weakly to a common measure $m\in \mathcal{P}_{2}$  and, furthermore, that there exist $\pi ^{\ast },U:\mathbb{R\times R\times \mathcal{P}}_{2}\times %
\left[ 0,T\right] \mathbb{\rightarrow R}$ such that, as $N\to \infty$, 
\begin{equation}\label{conjecture}
\begin{split}
&\pi ^{N,i,\ast }\left( x_{i},y,m^{N,i},t\right) \to \text{\ }\pi
^{\ast }\left( x,y,m,t\right),\\ 
&u^{N,i}\left(
x_{i},y,m^{N,i},t\right) \to \text{\ }U\left( x,y,m,t\right). 
\end{split}
\end{equation}%

Then, following similar approximations for each term in \eqref{HJB-finite}, we find that the formal limit, as $N\rightarrow \infty ,$ of  (\ref%
{HJB-finite})\ takes the form \eqref{master-intro} and the candidate mean field equilibrium optimal control $\pi
^{\ast }\left( x,y,m,t\right) $ is expected to satisfy \eqref{pi-intro}.

\vs
In summary, we expect the pair\text{\ }$\left( U\left( x,y,m,t\right) ,\pi
^{\ast }\left( x,y,m,t\right) \right) $ to satisfy the system of equations \eqref{master-intro} and \eqref{pi-intro}. 
\vs 
For the convenience of the reader we present a formal derivation of the optimality condition \eqref{pi-intro}. The master equation follows similarly.
\vs
 It follows from
\eqref{paris2} and \eqref{paris3} that, for $i\in\{1,\ldots, N\}$ and $j \in \{1,\dots, N\}\setminus \{i\}$, 
\[u^{N,i}_{x_ix_j}\simeq \frac{1}{N-1} u^N_{xm}(x_i,y,m^{N,i},x_j,t).\]
Inserting the above in \eqref{optN} and using \eqref{paris1} we find
\begin{equation*}
\begin{split}
&\pi^{N,\ast} u^N_{xx}(x_i,y,m^{N,i},t) \\
&+\frac{1}{N-1}\sum^N_{j=1,j\neq i} \pi^{N,\ast}(x_j,y,m^{N,i},t)u^N_{xm}(x_i,y,m^{N,i},x_j,t)\\
&=-bu^{N}_x(x_i,y, m^{N,i},t)-u^N_{xy}(x_i,y,m^{N,i},t).
\end{split}
\end{equation*}
After sending $N\to \infty$, the above expression yields formally  \eqref{pi-intro}.

\subsubsection{A simplified master equation}

The master equation \eqref{master-intro} can be simplified to a
similar equation corresponding to the case $b(y,t)\equiv 0$ and a modified
terminal utility. Indeed, in analogy to (\ref{u-w-F}), we set 
\begin{equation*}
W(x,y,m,t)=U\left( x,y,m,t\right) F(y,t), 
\end{equation*}%
with $F$ solving (\ref{F-equation}). Then, direct calculations in \eqref{master-intro} 
yield the master equation 
\be\label{master-reduced}
\begin{split}
&W_{t} +\frac{1}{2}\pi^{\ast 2} W_{xx} + \pi^{\ast }W_{xy} 
+\pi^{\ast } \mathcal{L}_1 W + \mathcal{L}_2 W + +\mathcal{L}_4 W\\ 
&\hskip.75in +\frac{1}{2} W_{yy}=0 \ \text{in} \ \R\times \R\times \mathcal{P}_2\times [0,T),\\[1.2mm]
&W(x,y,m,T)=J(x,m)F(y,T),
\end{split}
\ee
with $\mathcal{L}_1$,  $\mathcal{L}_2$ and $\mathcal{L}_4$ defined in  \eqref{takis1000.0}, \eqref{takis1000} and \eqref{takis1000.000} respectively.
\vs

Herein, we choose to work with the original version (\ref{master-intro}) for
two reasons. Firstly, \eqref{master-intro}  provides a natural object to use in
order to recover its analogue when $b(y,t)\equiv b>0,$ by sending $\nu $%
\textbf{\ }in (\ref{ass1}) to the Dirac function $\delta _{b}$. Secondly, 
it is expected that the results for non-zero drift can be used
to study MFG\ in complete markets with multi-dimensional local factors
(local volatility, predictable means, and others) and general utilities,
which arise in many applications but have not been studied so far.

\section{A separable class of MFG with partial information}

We study the master system \eqref{master-intro} and (\ref{pi-intro}) for
payoffs of the separable form, that is,%
\begin{equation}
J(x,m)=J\left( x-C\left( m\right) \right),
\label{payoff separable}
\end{equation}%
with $J$ being the utility of the representative agent and $C>0$ the
coupling function on the law of peers' wealth. We allow for general
dependence of $C$ on $m$, thus extending all works so far which considered only couplings
of form $C\left( m\right) =C(\bar{m})$. We examine the
latter special case in subsection 5.2 and subsection 5.3. We note that the minus sign
above is immaterial and is only used to allow for comparisons with existing works
in portfolio choice with relative performance. Indeed, one may allow for
homophilous interactions, in which players benefit from actions of their
peers; see \cite{Hu-Z} for an example with exponential utilities.

\subsection{Solution of the master system}

We construct a solution pair $U$ and $\pi ^{\ast }$ to system \eqref{master-intro} and (\ref{pi-intro}). 
\vs
Arguing formally, we  show that the value of the game can be represented as 
the value function $u$ of the single agent problem (\ref{u-DFN})\ (no
competition) with its argument translated by a function $f$ that solves a
non-local quasilinear  infinite-dimensional pde whose coefficients depend on the mean
field equilibrium feedback controls.  
We provide rigorous arguments when additional, but still general enough,
modeling assumptions are introduced.

\vs

We recall that the value function $u$ in (\ref{u-intro}) satisfies the terminal HJB equation \eqref{HJB-intro} and \eqref{takis1.11}
and that the optimal feedback control $\alpha^\ast$ is given by \eqref{feedback-intro}. 
%
\vs
Assume next that there exists a smooth solution $f\in \R\times\mathcal{P}_2\times [0,T]\to \R$ to 
\be\label{f-MFG-eqn}
\begin{split}
&f_{t}+ \mathcal{L}_2 f  + \mathcal{L}_4 f 
+\frac{1}{2}f_{yy}=0\text{ \ in \ }\mathbb{R\times }\mathcal{P}_{2} {
\times }\left[ 0,T\right],\\[1.5mm] 
&f(y,m,T)=C(m), 
\end{split}
\ee
with $\mathcal{L}_2$ and $\mathcal{L}_4$ as in \eqref{takis1000} and \eqref{takis1000.000} respectively, 
where  $\pi ^{\ast }:\R\times\R\times \mathcal{P}_{2}\times [0,T] \to \R$ satisfies, for each $(x,y,m,t) \in \mathbb{R\times }\mathcal{P}_{2}\mathcal{\times }\left[ 0,T\right]$, 
\begin{equation} \label{pi-f}
\begin{split}
\pi ^{\ast }(x,y,m,t)&- \ds \int \pi ^{\ast }(z,y, m, t)f_{m}(x,y,z,t)dm(z) \\
&=f_{y}(x,y,m,t) +a^{\ast }(x - \left( y,m,t\right) ,y,t). 
\end{split}
\end{equation}%
\vs

The regularity of $f$  is  important for what follows. We note, however, that it is not known   whether \eqref{f-MFG-eqn} has a smooth solution.  Later we give an interpretation to this 
equation and provide a formula for the possible solution. 
\vs
Next, we provide the main result of this section, representing the value $U$ of the game in terms of  the value function of the
problem of no competition and function $f$ above. We defer the interpretation of the solution to subsection~5.1 where we connect $U$ to the value of the writer of claim $C(m^\ast_T)$ (at the optimum,) $u$ being the value function of the ``plain'' investor and $f$ being the dynamic price of  $C(m^\ast_T)$.

\begin{proposition}
Let $u$ and $a^{\ast }$ be as in \eqref{HJB-intro}, \eqref{takis1.11}  and \eqref{feedback-intro}, and assume that \eqref{f-MFG-eqn} admits a smooth solution with $\pi^\ast$ satisfying \eqref{pi-f}. 
%
Then, a solution to  \eqref{master-intro} is given, for $(x,y,m,t) \in  \mathbb{R\times }\mathcal{P}_{2}\mathcal{\times }\left[ 0,T\right]$,  by 
\begin{equation} \label{U-u-f}
U(x,y,m,t)=u(x-f(y,m,t),y,t), 
\end{equation}%
with  $\pi^\ast$ satisfying \eqref{pi-f}. 
\end{proposition}

\begin{proof}
Assuming that $U$ and $F$ satisfy  \eqref{U-u-f} and \eqref{f-MFG-eqn}, we first show \eqref{pi-f}.
\vs
Differentiating  (\ref{U-u-f}), where for simplicity  we omit the dependence of $U, U_t, U_x,\\
U_y, U_{xy}, U_{xx}, U_{yy}$ on 
$(x,y,m,t)$, of $U_m, U_{mx}, U_{my}$ on $(x,y,m,z,t)$, of $U_{mm}$ on $(x,y,m,z_1,z_2,t)$, of $f, f_y, f_{yy}$ on $\left( y,m,t\right) $, of $f_m$ on  $( y,m, z, t)$, of $f_{mm}$ on $(y,m,z_1,z_2,t)$ and of $u$ and its derivatives on  
$(x-f(y,m,t), y,t)$,  yields 
\bee
\begin{split}
&U_{t}=-f_{t}u_{x}+u_{t}, \  U_{x}=u_{x}, \ %
U_{xx} =u_{xx}, \ U_{y}=-f_{y}u_{x}+u_{y},\\[1.5mm]
&  U_{yy}=-f_{yy}u_{x}+\left( f_{y}\right)
^{2}u_{xx} -2f_{y}u_{xy}+u_{yy}, \ U_{xy}=-f_{y}u_{xx}+u_{xy}, \\[1.5mm]
& U_{m}=-f_{m}u_{x}, \ U_{xm}=-f_{m} u_{xx}, \  U_{ym}=-f_{ym}u_{x} +f_{m}f_{y}u_{xx}-f_{m}u_{xy}, \\[1.5mm]
& U_{mm}=f_{m} f_{m} u_{xx}  +f_{mm} u_{x}, 
\end{split}
\eee

%

Inserting the above in  the optimality condition (\ref{pi-intro}) we find \eqref{pi-f}.
\vs
Next we show that $U$ defined in (\ref{U-u-f}) indeed satisfies the
master equation \eqref{master-intro}, as long as (\ref{f-MFG-eqn}%
)\ and (\ref{pi-f}) hold, and all involved partial derivatives exist.
\vs
Inserting the above derivatives in  \eqref{master-intro} and grouping
terms we find  
\begin{equation}\label{U-eqn-aux}
-u_{x}A +u_{t}+\frac{1}{2}u_{xx} B +\left(b u_{x}+u_{xy}\right) C+\frac{1}{2}%
u_{yy}+bu_{y}=0, 
\ee
where $u$ is evaluated at $\left( x-f(y,m,t),y,t\right) $ and the auxiliary
quantities $A, B, C:\R\times \mathcal{P}_2 \times [0,T]\to \R$ are given by 
\bee
\begin{split}
&A(y,m,t)=f_{t}\left( y,m,t\right) + \\
&\frac{1}{2}\int \int \pi ^{\ast }(z_{1},y,m,t)\pi ^{\ast
}(z_{2},y,m,t)f_{mm}(y,m,z_{1},z_{2},t)dm(z_{1})dm(z_{2})\\
&+\frac{1}{2}\int \left( \pi ^{\ast }(z,y,m,t)\right)
^{2}f_{mz}(y,m,z,t)dm(z)\\
&+\int \pi ^{\ast }(z,y,m,t)f_{my}(y,m,z,t)dm(z)+\frac{1}{2}f_{yy}\left(
y,m,t\right) , 
\end{split}
\eee
\bee
\begin{split}
&B(x,y,m,t)=\left( \pi ^{\ast }(x,y,m,t)\right) ^{2}\\
&-2\pi ^{\ast
}(x,y,m,t)\int \pi ^{\ast }\left( z,y,m,t\right) f_{m}\left(
x,y,m,z,t\right) dm\left( z\right) \\
&-2\pi ^{\ast }(x,y,m,t)f_{y}( y,m,t) +\Big( \int \pi ^{\ast
}(z,y,m,t)f_{m}(y,m,z,t)dm(z)\Big)^{2}\\
&+2f_{y}\left( m,y,t\right) \int \pi ^{\ast
}(z,y,m,t)f_{m}(y,m,z,t)dm(z)+\left( f_{y}\left( y,m,t\right) \right) ^{2} 
\end{split}
\eee
and 
\bee
\begin{split}
&C(x,y,m,t)=\pi ^{\ast }(x,y,m,t)\\
&-\pi ^{\ast }(x,y,m,t)\int \pi ^{\ast }\left( z,y,m,t\right) f_{m}\left(
x,y,m,z,t\right) dm\left( z\right) -f_{y}\left( y,m,t\right) . 
\end{split}
\eee
It follows that  $B=C^{2}$ which, together with (\ref{pi-f}), yields 
\begin{equation}
\frac{1}{2}u_{xx}B+\left(b u_{x}+u_{xy}\right) C=-\frac{\left( bu_{x}+u_{xy}\right) ^{2}}{2u_{xx}}. 
\label{HJB-aux}
\end{equation}%

Combining the above and using the HJB equation \eqref{HJB-single} satisfied by $u$, we get that  
\bee
u_{t}+\frac{1}{2}u_{xx}B 
+b u_{x}+u_{xy} +\frac{1}{2}u_{yy}+bu_{y}=0%
\text{ \ in \ }\mathbb{R\times R\times }\left[ 0,T\right] . 
\eee
Thus,  (\ref{U-eqn-aux})\ reduces to 
\[
u_{x} A=0\text{ \  \ in }%
\mathbb{R\times R\times }\left[ 0,T\right],
\]%
and the  claim follows, since  $u_{x} >0$ in $\mathbb{R\times R\times }\left[ 0,T\right] $.

\end{proof}

%

\section{Representative examples}

We analyze two examples. The first considers exponential
utilities and general couplings that are functions of the law of peers'
wealth. In such settings, the mean field equilibrium control becomes
independent of $x$ and (\ref{f-MFG-eqn}) reduces to the autonomous (\ref{f-exponential}). 
The value of the game is given by (\ref%
{U-MFG-expo}). If, in addition, the coupling depends only on the mean of
peers' wealth, equation (\ref{f-exponential}) reduces further to (\ref%
{f-expo-mean}), for which we provide a complete analysis, establishing
existence, uniqueness and regularity. We, also, produce in closed-form the
mean field equilibrium optimal processes in (\ref{X-optimal-expo})\ and (\ref%
{pi-optimal-expo}).
\vs
The second example considers general utilities and couplings that are linear
functions of peers' wealth. In such settings, the value of the game is given
by (\ref{U-linear}) and the mean field equilibrium processes are represented
as (\ref{X-MFG-process}) and (\ref{pi-MFG-process-linear}). The optimal
processes admit an intuitively pleasing decomposition in terms of the
corresponding optimal processes in the singe agent problem with modified
initial condition, and an additional component that involves averages with
regards to the initial distribution of peers' wealth.
\vs

We recall (\ref{X-representative}), which at a mean field equilibrium
control $\pi ^{\ast }$,  becomes 
\begin{equation}
dX_{s}^{\ast }=b(Y_{s},s)\pi _{s}^{\ast }ds+\pi _{s}^{\ast }dW_{s} \text{ \ and 
\ }X_{t}^{\ast }=x, \label{MFG-optimal-wealth}
\end{equation}%
and we  denote by $\bar{X}^{\ast }$ its conditional on the common noise average, 
\begin{equation}
\bar{X}_{s}^{\ast }=\int zdm_{s}^{\ast }\left( z\right) =\mathbb{E}\left[
\left. X_{s}^{\ast }\right \vert \mathcal{F}_{s}^{W}\right]  \ \text{in} \ [t,T] \text{\ and  \ }%
X_{t}^{\ast }=x,\text{\ }\bar{X}_{t}^{\ast }=\bar{m}, \label{MFG-average}
\end{equation}
with $x,$ $\bar{m}\in \mathbb{R},$ and $m\in \mathcal{P}_{2}$ being the
initial distribution of the players and $m^{\ast }$ the conditional on $%
\mathcal{F}^{W}$ law of $X^{\ast }$.

\subsection{Exponential utility and general couplings on the law of peers'
wealth}
As in subsection~2.4, here we consider the utility function 
$J(x)=-e^{-x}$ and recall
%
that the value function and the optimal policy of
the single player problem are given, for each $(x,y,t)\in \mathbb{R\times R\times }\left[ 0,T\right]$, by 
\[
u(x,y,t)=-e^{-x+k(y,t)}\text{ \ and \ }a^{\ast }(x,y,t)=c(y,t), 
\]%
with the functions $k$ and $c$ as  in \eqref{c-k.2} and (\ref{c-k.1}
). 
\vs
Then, from (\ref{pi-f}), we deduce that 
\[
\pi ^{\ast }(x,y,m,t)-\int \pi ^{\ast
}(z,y,m,t)f_{m}(y,m,z,t)dm(z)=f_{y}\left( y,m,t\right) +c(y,t), 
\]%
which yields that the candidate optimal feedback control is independent of $%
x,$ that is,
\[
\pi ^{\ast }(x,y,m,t)=\pi ^{\ast }(y,m,t),
\]%
and, thus,  
\begin{equation}\label{pi-MFG-expo}
\pi ^{\ast }(y,m,t)=\frac{f_{y}\left( y,m,t\right) +c(y,t)}{1-\int
f_{m}(y,m,z,t)dm(z)}. 
\end{equation}%
Then  (\ref{f-MFG-eqn})\ simplifies to 
\bee
f_{t} + \mathcal{L}_2 f   +\mathcal{L}_4f
+\frac{1}{2}f_{yy} =0\text{ \ in \ }\mathbb{R\times }\mathcal{P}_{2}\mathcal{%
\times }\left[ 0,T\right]. 
\eee
Therefore, using (\ref{pi-MFG-expo}), we obtain the autonomous terminal value problem
\be\label{f-exponential}
\begin{split}
&f_{t} +
\dfrac{1}{2} a^2 \hat{ \mathcal{L}_2} (f)
+ a \hat{\mathcal{L}}_4 (f)
+\frac{1}{2}f_{yy} =0\text{\  \ in \ }\mathbb{R\times }%
\mathcal{P}_{2}\mathcal{\times }\left[ 0,T\right],\\ 
&f(y,m,T)=C(m),
\end{split}
\ee
where
\[a(y,m,t)= \dfrac{f_{y}(y,m,t)  
+c(y,t)}{1-\int f_{m}(y,m,z,t)dm(z)},\]
\[\hat{\mathcal{L}}_2(f)(y,m,t)= \int  \int 
f_{mm}(y,m,z_{1},z_{2},t)dm(z_{1})dm(z_{2}) + \int f_{mz}(y,m,z,t)dm(z),\]
and
\[\ds \hat{\mathcal{L}}_4(f)(y,m,t)=\int f_{ym}(y,m,z,t) dm(z).\]
\vs
The value of
the game is given, for each $(x,y,t) \in \R\times\R\times [0,T]$,   by 
\begin{equation}\label{U-MFG-expo}
U(x,y,m,t)=-e^{-(x-f(y,m,t))+k(y,t)}.%
\end{equation}%


\subsection{Exponential utility and general couplings on the mean of peer's
wealth}

We consider coupling functions $C:\mathcal{P}_2 \to \R$ that depend only on the mean of peers'
wealth, that is, 

\begin{equation}
C(m)=C\left( \bar{m}\right). 
\label{Coupling-mean}
\end{equation}%

The special case $C\left( \bar{m}\right) =\theta \bar{m},$ $\bar{m}\in 
\mathbb{R},$ with $\theta \in \left( 0,1\right) $, is the only one so far
analyzed in the literature under exponential utility. The general case $C(%
\bar{m})$ was recently solved in \cite{Souganidis-Z-MFG} for Black and
Scholes markets and, herein, we extend these results for the single factor
model (\ref{stock-reduced}) and (\ref{Y-reduced}).

\vs
As far as the coupling is concerned, we assume that $C:\mathbb{%
R\rightarrow R}$ satisfies (\ref{Coupling-mean}) with 
\be\label{Assumption-C}
\begin{array}{c}
C\in \mathcal{C}^{2}( \mathbb{R}),\ C(0)=0,\\
\text{ and, 
for some constants $k_{1},k_{2}, K, L>0$ and all $z\in \R$}, \\ 
k_{1}<1-C^{\prime }\left( z\right) <k_{2}\text{ \ and \ }\left( z-C\left(
z\right) \right) ^{(-1)}\leq Ke^{Lz^{2}}.%
\end{array}
\end{equation}%
We note that there is no assumption of monotonicity on $C,$ which is one
of the key assumptions in many references in the MFG\ literature.
\vs
In the sequel we provide a complete analysis of the value of the game,
the mean field equilibrium feedback control, and the optimal processes.
\vs
We
start with  (\ref{f-exponential})\ and find an $x-$independent smooth solution $f:\R\times \mathcal{P}_{2} \times [0,T]\to \R$
to \eqref{f-MFG-eqn}.
\vs
To this end, we use $c$  given by \eqref{c-k.1} 
and  introduce the auxiliary smooth function 
\begin{equation}
q(y,t)=\left \{ 
\begin{array}{c}
-\frac{1}{2}\int_{0}^{t}c_{y}\left( 0,s\right) ds+\int_{0}^{y}c(\rho
,t)d\rho \text{ \ in \ }(0,\infty)\mathbb{\times }\left[ 0,T\right], \\ 
\\
\hskip-.65in-\frac{1}{2}\int_{0}^{t}c_{y}\left( 0,s\right) ds \ \text{if } \ (y,t)\in \{0\} \times [0,T],\\
\\
\hskip.075in -\frac{1}{2}\int_{0}^{t}c_{y}\left( 0,s\right) ds-\int_{y}^{0}c(\rho
,t)d\rho \text{ \  \ in \  \ }(-\infty, 0)\mathbb{\times }\left[ 0,T\right]. %
\end{array}%
\right. \text{ \  \  \  \ }  \label{n(y,t)-aux}
\end{equation}%

\begin{proposition}
Assume \eqref{Coupling-mean} and \eqref{Assumption-C}. Then, \eqref%
{f-exponential} reduces to 
\begin{equation}\label{f-expo-mean}
\begin{split}
&f_{t} +\frac{1}{2}\left( \frac{f_{y} +c}{1-f_{\bar{m}}}\right) ^{2}f_{\bar{m}\bar{m}%
}
+\frac{f_{y}+c}{1-f_{\bar{m}}}%
f_{y\bar{m}}
+\frac{1}{2}f_{yy} =0%
\text{ \ in \ }\mathbb{R\times R}\mathcal{\times }\left[ 0,T\right),\\[1.5mm]
&f\left( y,\bar{m},T\right) =C\left( \bar{m}\right), 
\end{split}
\ee
which has a unique smooth solution given, for $(y,\bar m,t) \in \R\times\R\times [0,T]$, by 
\begin{equation}
f(y,\bar{m},t)=\bar{m}-\left( g\left( y,\cdot ,t\right) \right) ^{\left(
-1\right) }(y,\bar{m},t)-q(y,t),  \label{g-f-expo}
\end{equation}%
where $g$ solves, for $q$ as in (\ref{n(y,t)-aux}) and each $\bar{m}\in \mathbb{R}$, 
\begin{equation}\label{g-f-expo-heat}
\begin{cases}
&g_{t}+\frac{1}{2}g_{yy}=0\text{ \ in \ }\mathbb{R}\mathcal{\times }\left[
0,T\right),\\[1.5mm]
&g(y,\bar{m},T)=\left( \cdot -C(\cdot )-q(y,T),y,T\right) ^{\left( -1\right)
}\left( \bar{m},y,T\right). 
\end{cases}
\ee
Furthermore, 
\begin{equation} \label{1-f/m}
1-f_{\bar{m}}>0 \ \text{in} \  \mathbb{R\times R}\mathcal{%
\times }\left[ 0,T\right].
\end{equation}
\end{proposition}

\begin{proof}
We show first that (\ref%
{f-expo-mean}) can be reduced to (\ref{h-expo-aux}) below, which corresponds
to the case $c(y,t)=0$ in $\mathbb{R}\mathcal{\times }\left[ 0,T\right] .$
Then, we construct a unique smooth solution to (\ref{h-expo-aux}) and define $\hat{%
f}$ as in (\ref{f-expo-aux}). In turn, we show that $\hat{f}$ solves
(\ref{f-expo-mean}) with $\hat f\left( y,\bar{m},T\right) =C(\bar{m})$,
and that it is smooth. Thus, by the uniqueness of viscosity solutions, it
coincides with $f.$
\vs

Let $q(y,t)$ be as in (\ref{n(y,t)-aux}), and consider the solution $h=h(y,\bar m,t)$ to 
\begin{equation}\label{h-expo-aux}
\begin{split}
&h_{t} +\frac{1}{2}\left( \frac{h_{y}}{1-h_{\bar{m}}}\right) ^{2}h_{\bar{m}\bar{m}} +\frac{h_{y}}{1-h_{\bar{m}}}h_{y\bar{m
}}+\frac{1}{2}h_{yy} =0 \text{ \ in \ }%
\mathbb{R\times R\times }\left[ 0,T\right],\\[1.5mm]
&h\left( y,\bar{m},T\right) =C(\bar{m})+q(y,T).
\end{split}
\end{equation}%
We assume for now that  \eqref{h-expo-aux} has a unique smooth solution, a fact that we will
verify at the end of the proof, and we introduce $\hat f:\mathbb{R\times R\times }\left[ 0,T\right] \to \R$ given by 
\begin{equation}
\hat{f}\left( y,\bar{m},t\right) =h\left( y,\bar{m},t\right) -q(y,t)\text{ \
in \ }\mathbb{R\times R\times }\left[ 0,T\right],   \label{f-expo-aux}
\end{equation}%
and claim that it solves (\ref{f-expo-mean}).

\vs
Since  $\hat f(y,\bar m,T)=C(\bar m)$ is obvious, we  show that 
it solves the pde in (\ref{f-expo-mean}).
\vs

We first recall that, since $c\left( y,t\right) $ solves (\ref%
{c-heat-eqn}), we have 
\[
\int_{0}^{y}c_{t}(\rho ,t)d\rho =-\frac{1}{2}c_{y}(y,t)+\frac{1}{2}c_{y}(0,t)%
\text{ \ in \ }(0,\infty)\mathcal{\times }\left[ 0,T\right] , 
\]%
and, similarly, 
\[
\int_{y}^{0}c_{t}(\rho ,t)d\rho =-\frac{1}{2}c_{y}(y,t)+\frac{1}{2}c_{y}(0,t)%
\text{ \ in \ }(-\infty,0)\mathcal{\times }\left[ 0,T\right] . 
\]%
Then, if $y\geq 0$, 
\bee
 \hat{f}_{t}=h_{t} +\frac{1}{2}c_{y}, \ \hat{f}_{\bar{m}} =h_{\bar{m}}, \ \hat{f}_{y} =h_{y} -c,  \  \hat{f}_{\bar{m}\bar{m}} = h_{\bar{m}\bar{m}}. 
\eee
Inserting  the above derivatives in (\ref{h-expo-aux}) yields that $\hat f$ solves \eqref{f-expo-mean} in $[0,\infty)\times \R\times [0,T)$.  A similar calculation yields the claim in $(-\infty,0]\times \R\times [0,T)$.
%
\vs
To  prove that $h$ is the unique smooth solution to the terminal
value problem (\ref{h-expo-aux}), 
we define, for each $(y,t)\in \R\times [0,T]$, the map $g(y,\cdot,t):\R\to \R$ by  
\[
g(y,\bar{m},t)=\left( \cdot -h\left( y,\cdot ,t\right) \right)
^{(-1)}\left( y,\bar{m},t\right) \text{\  \ in \ }\mathbb{R\times R\times }%
\left[ 0,T\right], 
\]%
and claim that, for each $\bar{m}\in \mathbb{R}$, $g(\cdot,\bar m,\cdot) $ is a solution to 
\begin{equation}
\begin{split}
&g_{t}+\frac{1}{2}g_{yy}=0\text{ \ in }\mathbb{R\times }\left[ 0,T\right),\\
& g\left( y,\bar{m},T\right) =\left( \cdot -F(\cdot )-p(y,T),y,T\right)
^{\left( -1\right) }\left( y,\bar{m},T\right).
\end{split}
\label{g-heat}
\end{equation}%
Since the terminal condition follows trivially, we only need to show that $g$ satisfies the equation in (%
\ref{g-heat}).
\vs
Let $N:\R\times\R\times [0,T]\to \R$ be defined as
\begin{equation}\label{N-h}
N(y,\bar{m},t)=\bar{m}-h(y,\bar{m},t).
\end{equation}%
Then, $N$ solves the terminal value problem 
\begin{equation}\label{N-equation}
\begin{split}
&N_{t} +\frac{1}{2}\left( \dfrac{N_{y}}{N_{\bar{m}}}\right) ^{2} N_{\bar{m}\bar{m}%
}
-\dfrac{N_{y}}{N_{\bar{m}}}%
N_{y\bar{m}} +\frac{1}{2}N_{\bar{m}\bar{m}} =0\text{ \  in \ }\mathbb{R\times }\R%
\mathcal{\times }\left[ 0,T\right),\\[2mm]
& N(y,\bar{m},T)=\bar{m}-F(\bar{m})+c(y,T).
\end{split}
\ee

The definitions of $h$ and $g$ imply   that $N\left( y,g(y,\bar{m}
,t),t\right) =\bar{m}$. In turn, 
%
evaluating equation (\ref{N-equation})\ at $g(y,%
\bar{m},t)$ and using the above yields%
\[
-\frac{g_{t}}{g_{\bar{m}}}+\frac{1}{2}\left( \left( g_{y}\right) ^{2}N_{_{%
\bar{m}\bar{m}}}+2g_{y}N_{_{\bar{m}}y}+N_{yy}\right) =-\frac{1}{g_{\bar{m}}}\left( g_{t}+\frac{1}{2}g_{yy}\right) =0. 
\]%

To conclude, we consider the terminal value problem (\ref{g-heat}%
) with $q(y,T)$ as in (\ref{n(y,t)-aux}) for $%
t=T.$ Then, $g$ is smooth and, furthermore, the maximum principle yields $g_{%
\bar{m}} >0$ in $\mathbb{R\times R}\mathcal{\times }%
\left[ 0,T\right] .$ Thus, $g$ is invertible in $\bar{m}$,  and its  inverse $%
g\left( y,\cdot ,t\right) ^{\left( -1\right) }\left( y,\bar{m},t\right)
=N\left( y,\bar{m},t\right) $ is well defined and smooth. In turn, (\ref{N-h}%
)\ yields that $h$ is well defined and smooth, and \eqref{g-f-expo} follows. 
\vs

The rest of the proof follows.


\end{proof}

Next, we construct  the mean field equilibrium processes, $%
X^{\ast }$ and $\pi ^{\ast }.$

\begin{proposition}
Let $f$ as in Proposition~12,  $k$ as in (\ref{c-k.2}) and
introduce the map $G: \mathbb{R\times R\times }\left[ 0,T\right]$ given  by
\begin{equation}
G(y,\bar{m},t)=\left( \cdot -f(y,\cdot ,t\right) )^{\left( -1\right) }\left(
y,\bar{m},t\right).
 \label{g-^}
\end{equation}%
The mean field value is given by 
\begin{equation}
U(x,y,m,t)=-e^{-(x-f(y,\bar{m},t))+k(y,t)}\text{ \  \ in \ }\mathbb{R\times
R\times }\left[ 0,T\right] .  \label{U-expo}
\end{equation}%
The mean field equilibrium process, $X^{\ast },$ is represented, for $0\leq t\leq s\leq T$, as 
\begin{equation}
X_{s}^{\ast }=x-f\left( y,\bar{m},t\right) +L_{t,s} +f\left( Y_{s},G(Y_{s},x-f\left( y,\bar{m},t\right) +L_{t,s},s),s\right), \label{X-optimal-expo}
\end{equation}%
with the process $L$ defined in (\ref{L-intro}). Alternatively,  for $0\leq t\leq s\leq T$,%
\begin{equation} \label{X-optimal-X-single-expo}
\begin{split}
X_{s}^{\ast }&=\mathcal{X}_{s}^{x-f\left( y,\bar{m},t\right) ,\ast }+f\left(
Y_{s},G(Y_{s},\mathcal{X}_{s}^{x-f\left( y,\bar{m},t\right) ,\ast },s\right)
,s)\\
& =\mathcal{X}_{s}^{x-f\left( \bar{m},y,t\right) ,\ast }+\mathcal{E}\left(
Y_{s},\mathcal{X}_{s}^{x-f\left( \bar{m},y,t\right) ,\ast },s\right),
\end{split}
\ee %
where, for  $(y,\bar{m},t) \in \mathbb{R\times R\times }\left[ 0,T\right]$, 
\begin{equation}
\mathcal{E}\left( y,\bar{m},t\right) =f(y,G(y,\bar{m},t),t) ,  \label{E-nonlinear}
\end{equation}%
and $\mathcal{X}^{x-f\left( y,\bar{m},y,t\right) ,\ast }$ is  the optimal
wealth in the single agent problem (\ref{u-DFN}), starting at $x-f\left( y,%
\bar{m},t\right) $; see  (\ref{x-thm}). 
\vs
Moreover, 
\[
\bar{X}_{s}^{\ast }=G(Y_{s},x-f\left( y,\bar{m},t\right) +L_{t,s},s)\text{ \
in \ }\mathbb{R\times R\times }\left[ t,T\right] ,
\]%
and the mean field equilibrium control process, $\pi ^{\ast },$ is given, $0\leq t\leq s\leq T$, by%
\begin{equation}\label{pi-optimal-expo}
\pi _{s}^{\ast }=\frac{f_{y}\left( Y_{s},\bar{X}_{s}^{\ast },s\right)
+c(Y_{s},s)}{1-f_{\bar{m}}\left( Y_{s},\bar{X}_{s}^{\ast },s\right) }.
\end{equation}
\end{proposition}

\begin{proof}
From (\ref{pi-MFG-expo}) we have 
\[
\pi _{s}^{\ast }=\pi ^{\ast }\left( Y_{s},\bar{X}_{s}^{\ast },s\right) =%
\frac{f_{y}\left( Y_{s},\bar{X}_{s}^{\ast },s\right) +c(Y_{s},s)}{1-f_{\bar{m%
}}\left( Y_{s},\bar{X}_{s}^{\ast },s\right) },
\]%
with $\bar{X}_{t}^{\ast }=\bar{m},$ which is well defined in view of (%
\ref{1-f/m}). Using the above and (\ref{MFG-optimal-wealth}) yields, for $%
0\leq t\leq s\leq T,$ 
\[
dX_{s}^{\ast }=\frac{f_{y}\left( Y_{s},\bar{X}_{s}^{\ast },s\right)
+c(Y_{s},s)}{1-f_{\bar{m}}\left( Y_{s},\bar{X}_{s}^{\ast },s\right) }\left(
b(Y_{s},s)ds+dW_{s}\right) ,
\]%
with $X_{t}^{\ast }=x$ and, thus, 
\begin{equation} \label{X-average-dynamics}
d\bar{X}_{s}^{\ast }=\frac{f_{y}\left( Y_{s},\bar{X}_{s}^{\ast },s\right)
+c(Y_{s},s)}{1-f_{\bar{m}}\left( Y_{s},\bar{X}_{s}^{\ast },s\right) }\left(
b(Y_{s},s)ds+dW_{s}\right) \ \text{in} \ (t,T) \ \ \bar{X}_{t}^{\ast }=\bar{m}.\\
\end{equation}%
Applying It\^{o}'s formula to $f\left( Y_{s},%
\bar{X}_{s}^{\ast },,s\right) $ and using (\ref{f-expo-mean}) gives%
\bee
\dd f\left( Y_{s},\bar{X}_{s}^{\ast },s\right)=f_{\bar{m}}\left( Y_{s},\bar{X}_{s}^{\ast },s\right) d\bar{X}_{s}^{\ast}
+f_{y}\left( Y_{s},\bar{X}_{s}^{\ast },s\right) dY_{s}.
\eee

Using (\ref{X-average-dynamics}%
) and regrouping terms we obtain 
\[
\dd f\left( Y_{s},\bar{X}_{s}^{\ast },s\right) =b(Y_{s},s)\Big( \frac{%
f_{y}\left( \bar{X}_{s}^{\ast },Y_{s},s\right) +f_{\bar{m}}\left( \bar{X}%
_{s}^{\ast },Y_{s},s\right) c(Y_{s},s)}{1-f_{\bar{m}}\left( \bar{X}%
_{s}^{\ast },Y_{s},s\right) }\Big) ds
\]%
\[
+\frac{f_{y}\left( \bar{X}_{s}^{\ast },Y_{s},s\right) +f_{\bar{m}}\left( 
\bar{X}_{s}^{\ast },Y_{s},s\right) c(Y_{s},s)}{1-f_{\bar{m}}\left( \bar{X}%
_{s}^{\ast },Y_{s},s\right) }dW_{s}.
\]%
Therefore, with  the process $L$ as  in (\ref{L-intro}),%
\[
d(X_{s}^{x,\ast }-f\left( Y_{s},\bar{X}_{s}^{\ast },s\right) )=b(Y_{s},s)c(Y_{s},s)ds+c\left( Y,s\right) dW_{s}=dL_{t,s}.
\]%
It follows that 
\[
X_{s}^{\ast }-f\left( Y_{s},\bar{X}_{s}^{\ast },s\right) =x-f\left( y,\bar{m}%
,t\right) +L_{t,s},
\]%
and, in turn, 
\[
\bar{X}_{s}^{\ast }=G(Y_{s},\bar{m}-f\left( y,\bar{m},t\right) +L_{t,s},s).
\]%
To show (\ref{X-optimal-X-single-expo}), we use (\ref%
{E-nonlinear}) and recall (\ref{X-a-optimal-expo-Merton}).

\end{proof}

\subsection{General utility and couplings linear in the mean of peers' wealth%
}

We consider general utilities $J$ satisfying \eqref{utility} and linear couplings, that is, for each $m\in \mathcal{P}_2$ and 
$\theta \in (0,1)$, 
\begin{equation} \label{C-linear}
C(m)=\theta \bar{m}.
\end{equation}%
First, we observe that  
\begin{equation}
f\left( y,\bar{m},t\right) =\theta \bar{m}, 
 \label{f-linear}
\end{equation}%
is a smooth solution to (\ref{f-MFG-eqn}) and, since  $1-f_{\bar{m}}(y,m,t)=1-\theta >0$ in 
$\mathbb{R\times R\times }\left[ 0,T\right] ,$  it is the unique
smooth solution. 
\vs
Since  the optimality condition (\ref{pi-f}) reduces to%
\[
\pi ^{\ast }(x,y,m,t)-\theta \int \pi ^{\ast }(z,y,m,t)dm(z)=\alpha ^{\ast
}(x-\theta \bar{m},y,t), 
\]%
with $\alpha ^{\ast }(x,y,t)$ as in \eqref{a-R-thm}, it follows that 
\[
\int \pi ^{\ast }(z,y,m,t)dm(z)=\frac{1}{1-\theta }\int \alpha ^{\ast
}(z-\theta \bar{m},y,t)dm(z), 
\]%
and, thus, for each $(x,y,m,t) \in \R\times \R\times \mathcal{P}_2 \times \left[ 0,T\right]$,
\begin{equation}
\pi ^{\ast }\left( x,y,m,t\right) =\alpha ^{\ast }\left( x-\theta \bar{m},y,t\right) +\frac{\theta }{1-\theta }%
\int \alpha ^{\ast }\left( z-\theta \bar{m},y,t\right) dm\left( z\right). 
\label{pi-linear-feedback}
\end{equation}%
Then,   with $u$ as in (\ref{u-DFN}) and using  (\ref{U-u-f}) and (\ref{f-linear}), we obtain%
\begin{equation}
U(x,y,m,t)=u\left( x-\theta \bar{m},y,t\right).  \label{U-linear}
\end{equation}%

To construct the optimal mean field and portfolio processes, we note that it follows 
from (\ref{MFG-optimal-wealth}) and (\ref{pi-linear-feedback}), 
that $X^{\ast }$ solves 
\[
dX_{s}^{\ast }=b(Y_{s},s)\left( \alpha ^{\ast }(X_{s}^{\ast }-\theta \bar{X}%
_{s}^{\ast },Y_{s},s)+\frac{\theta }{1-\theta }\int \alpha ^{\ast }(z-\theta 
\bar{X}_{s}^{\ast },Y_{s},s)dm_{s}^{\ast }(z)\right) ds 
\]%
\[
+\left( \alpha ^{\ast }(X_{s}^{\ast }-\theta \bar{X}_{s}^{\ast },Y_{s},s)+%
\frac{\theta }{1-\theta }\int \alpha ^{\ast }(z-\theta \bar{X}_{s}^{\ast
},Y_{s},s)dm_{s}^{\ast }(z)\right) dW_{s},\text{ \ } 
\]%
with $X_{t}^{\ast }=x$ and $\bar{X}^{\ast }$ as in (\ref{MFG-average}).
Therefore,%
\[
d\bar{X}_{s}^{\ast }=\frac{1}{1-\theta }b(Y_{s},s)\left( \int \alpha ^{\ast
}(z-\theta \bar{X}_{s}^{\ast },Y_{s},s)dm_{s}^{\ast }(z)\right) ds 
\]%
\[
+\frac{\theta }{1-\theta }\left( \int \alpha ^{\ast }(z-\theta \bar{X}%
_{s}^{\ast },Y_{s},s)dm_{s}^{\ast }(z)\right) dW_{s},\text{ \ } 
\]%
with \ $\bar{X}_{t}^{\ast }=\bar{m}$. 
\vs
It follows that $X_{t}^{\ast }-\theta \bar{X}_{t}^{\ast }=x-\theta \bar{m}$ and, for $0\leq t\leq s\leq T$,%
\bee\begin{split}
&\dd \left( X_{s}^{\ast }-\theta \bar{X}_{s}^{\ast }\right) \\
&=b(Y_{s},s)\alpha ^{\ast }(X_{s}^{\ast }-\theta \bar{X}_{s}^{\ast
},Y_{s},s)ds+\alpha ^{\ast }(X_{s}^{\ast }-\theta \bar{X}_{s}^{\ast
},Y_{s},s)dW_{s}.
\end{split}
\eee

Working as in the proof in Theorem~7, we obtain that the sde above, which is
autonomous in the argument $X^{\ast }-\theta \bar{X}^{\ast },$ has a unique
strong solution, given by 
\begin{equation}
X_{s}^{\ast }-\theta \bar{X}_{s}^{\ast }=H\left( H^{\left( -1\right) }\left(
x-\theta \bar{m},y,t\right) +L_{t,s},Y_{s},s\right),
\label{X-X(ave)}
\end{equation}%
with the process $L$ as in (\ref{L-intro}) and the function $H$ solving (\ref%
{H-intro}).
\vs
Next, averaging with regards to $m^{\ast }$ and using the separability of
the initial condition $x-\theta \bar{m}$ from the common noise in the process $%
X^{\ast }-\theta \bar{X}^{\ast }$ above, we obtain, with  $m_{t}^{\ast }=m$, that 

\begin{equation}
\bar{X}_{s}^{\ast }=\frac{1}{1-\theta }\int H\left( H^{\left( -1\right)
}\left( x-\theta \bar{m},y,t\right) +L_{t,s},Y_{s},s\right) dm(x). 
\label{X-ave-linear}
\end{equation}%
The representation (\ref{x-thm})\ for the
optimal wealth in the single agent optimization problem (\ref{u-DFN}) implies that, for 
$s\in [t,T]$,
\begin{equation}
\int H\left( H^{\left( -1\right) }\left( x-\theta \bar{m},y,t\right)
+L_{t,s},Y_{s},s\right) dm(x) =\dint \mathcal{X}_{s}^{x-\theta \bar{m},\ast }dm(x).
\label{X-X(ave)-reduction}
\end{equation}%
Combining \eqref{X-X(ave)}, (\ref{X-ave-linear}) and (\ref{X-X(ave)-reduction}%
), we obtain that the mean field equilibrium wealth process is given by 
\bee
\begin{split}
X_{s}^{\ast }&=H\left( H^{\left( -1\right) }\left( x-\theta \bar{m}%
,y,t\right) +L_{t,s},Y_{s},s\right)\\
&+\frac{\theta }{1-\theta }\int H\left( H^{\left( -1\right) }\left( x-\theta 
\bar{m},y,t\right) +L_{t,s},Y_{s},s\right) dm(x). 
\end{split}
\eee
In other words, reinstating the appropriate initial conditions we have, for \\ 
$0\leq t\leq s\leq
T$, the
decomposition
\begin{equation}
X_{s}^{x,\ast }=\mathcal{X}_{s}^{x-\theta \bar{m},\ast }+\theta \bar{X}_{s}^{%
\bar{m},\ast }=\mathcal{X}_{s}^{x-\theta \bar{m},\ast }+\frac{\theta }{1-\theta }\dint 
\mathcal{X}_{s}^{x-\theta \bar{m},\ast }dm(x). \label{X-MFG-process}
\end{equation}%
To construct the associated mean field equilibrium control process, we
observe that, for $0\leq t\leq s\leq
T$, %
\begin{equation}
\begin{split}
&\pi _{s}^{\ast }=\pi ^{\ast }\left( X_{s}^{\ast },Y_{s},m_{s}^{\ast
},s\right)\\
&=\alpha ^{\ast }\left( X_{s}^{\ast }-\theta \bar{X}_{s}^{\ast
},Y_{s},s\right) +\frac{\theta }{1-\theta }\int \alpha ^{\ast }\left(
z-\theta \bar{X}_{s}^{\ast },Y_{s},s\right) dm_{s}^{\ast }\left( z\right).%
 \label{pi-MFG-process-linear}
\end{split}
\end{equation}%
From (\ref{X-X(ave)})\ and the second equality in (\ref{a-process-thm}), we
obtain 
\[
\alpha ^{\ast }\left( X_{s}^{\ast }-\theta \bar{X}_{s}^{\ast
},Y_{s},s\right) =\alpha ^{\ast }\left( H\left( H^{\left( -1\right) }\left(
x-\theta \bar{m},y,t\right) +L_{t,s},Y_{s},s\right) ,Y_{s},s\right) 
\]%
\[
=\alpha ^{\ast }\left( \mathcal{X}_{s}^{x-\theta \bar{m},\ast
},Y_{s},s,\right) =\alpha _{s}^{x-\theta \bar{m},\ast }, 
\]%
the optimal process of the single agent problem (\ref{u-DFN})
starting at $\left( x-\theta \bar{m},y,t\right) $. 
\vs
Thus, 
\bee
\int \alpha ^{\ast }\left( z-\theta \bar{X}_{s}^{\ast },Y_{s},s\right)
dm_{s}^{\ast }\left( z\right) =\int \alpha ^{\ast }\left( \mathcal{X}%
_{s}^{x-\theta \bar{m},\ast },Y_{s},s\right) dm(x)=\int \alpha
_{s}^{x-\theta \bar{m},\ast }dm(x),
\eee
%
%
%
and conclude that the mean field equilibrium policy is represented as 
\begin{equation}
\pi _{s}^{x,\ast }=\alpha _{s}^{x-\theta \bar{m},\ast }+\frac{\theta }{%
1-\theta }\int \alpha _{s}^{x-\theta \bar{m},\ast }dm(x).
\label{pi-MFG-process}
\end{equation}


We can interpret the above optimal processes as
follows. At time $t,$ the representative agent starts at $x,$ and splits it
to $x_{1}=x-\theta \bar{m}$ and to $x_{2}=\theta \bar{m}$. Using $x_{1},$
she follows the optimal policy $\alpha ^{x_{1},\ast }$ as in  (\ref%
{u-DFN}), which generates the process 
\[
X_{1,s}^{x_{1},\ast }=H\left( H^{\left( -1\right) }\left( x_{1},y,t\right)
+L_{t,s},Y_{s},s\right). 
\]%
With the remaining $x_{2},$ she follows policy $\frac{\theta }{1-\theta }%
\int \alpha _{s}^{x-\theta \bar{m},\ast }dm(x),$ obtaining%
\[
X_{2,s}^{x_{2},\ast }=\frac{\theta }{1-\theta }\int \alpha ^{\ast }\left(
H\left( H^{\left( -1\right) }\left( x_{1},y,t\right) +L_{t,s},Y_{s},s\right)
,Y_{s},s\right) dm(x) 
\]%
\[
=\frac{\theta }{1-\theta }\dint \mathcal{X}_{s}^{x-\theta \bar{m},\ast
}dm(x). 
\]%
We note that $\frac{\theta }{1-\theta }\int \alpha _{s}^{x-\theta \bar{m}%
,\ast }dm(x)$ is a feasible policy but not optimal.

\subsection{Separable payoffs and connection with indifference valuation and
arbitrage-free pricing}

We conclude by relating and interpreting the results for separable payoffs
using elements from indifference valuation and  arbitrage-free pricing. For the
reader's convenience, we recall the notion of indifference price which, to
avoid cumbersome notation, we present somewhat informally (we refer the
reader to \cite{carmo} and \cite{Musiela-Z-indifferencePricing} for a general overview of the area). 
\vs
Let $Z_{T}$ be a given claim, represented as an $\mathcal{F}_{T}^{Y}-$%
measurable random variable and introduced at $t$. Then, its so-called
writer's value function is defined as 
\[
w^{Z}(x,y,t)=\sup_{\mathcal{A}}\mathbb{E}_{\mathbb{P}}\left[ \left. J(%
\mathcal{X}_{T}-Z_{T})\right \vert \mathcal{X}_{t}=x,Y_{t}=y,\mathcal{F}%
_{t}^{Y}\right] , 
\]%
with $\mathcal{X}$ and $Y$ solving (\ref{takis2000}) and (\ref%
{Y-reduced}). 
\vs
At time $t,$ the indifference price of $Z_{T}$, denoted by $%
p_{t}(Z_{T}),$ is defined as the spatial input $p_{t}(Z_{T})\in \mathcal{F}%
_{t}^{Y}$ such that, a.s. in $\mathbb{R\times
R\times }\left[ 0,T\right]$, 
\[
w^{0}(x,y,t)=w^{Z}(x+p_{t}(Z_{T}),y,t).
\]%
In other words, $p_{t}(Z_{T})$ makes the agent "indifferent" between
(i) investing optimally in the underlying market without the claim ($%
Z_{T}\equiv 0)$ and (ii) receiving compensation $p_{t}(Z_{T})$ at initiation
time $t,$ investing optimally in the same market and rendering $Z_{T}$ at
expiry time $T$. 
\vs
We are also interested in the indifference price process $%
p_{s}\left( Z_{T}\right) ,$ $0\leq t\leq s\leq T,$ and the sde\ satisfied by
it, with $p_{T}(Z_{T})=Z_{T}.$
\vs
In the general case of incomplete markets, calculating the indifference
price process is a taunting task as the process $p_{s}(Z_{T})$ is, in
general, path-dependent and incorporates in a complicated way both the
market dynamics and the writer's wealth.
\vs
If, however, the market is complete, as the one herein, and under mild
integrability assumptions on $Z_{T},$ it can be shown that its indifference
price process coincides with its so-called arbitrage-free price,
which is given, for $0\leq t\leq s\leq T$, by the martingale \ 
\be\label{arbitrage}
p_{s}(Z_{T})=\mathbb{E}_{\mathbb{Q}}[\left. Z_{T}\right \vert \mathcal{F}%
_{s}^{Y}]. 
\ee

The measure $\mathbb{Q}$ is the unique, absolutely continuous to the
(physical) measure $\mathbb{P}$, that makes the (discounted) stock price a
martingale\footnote{%
The reader familiar with arbitrage-free pricing may see the connection
between the pricing measure $\mathbb{Q}$ and the one appearing in the
dynamics \eqref{X-Y-tilda} of the auxiliary problem (\ref{u-tilda}).}. It then follows, from the
martingale representation theorem and the definition of $\mathbb{Q},$ that we can 
obtain the decomposition 
\[
dp_{s}(Z_{T})=h_{s}dW_{s}^{\mathbb{Q}}=bh_{s}ds+h_{s}dW_{s} \ \text{for} \ s\in(t,T)
\ \text{ and} \ 
p_{T}(Z_{T})=Z_{T},\]%
with $W$ as in (\ref{W-innovations}) and some process $h$, known as the
"indifference hedge". 
\vs
Furthermore, it can be shown that the above sde\
also coincides with the actual under $\mathbb{P}$ dynamics that generate
claim $Z_{T},$ that is, we have  $p_{s}(Z_{T})=Z_{s},$ or
equivalently, if the process $Z$ solves 
\[
dZ_{s}=b(Y_{s},s)h_{s}ds+h_{s}dW_{s},\text{ }Z_{t}=z,\text{ \ }0\leq t\leq
s\leq T,
\]%
then 
\[p_{s}(Z_{T})=Z_{s} \ \text{ in} \ (t, T) \ \text{and} \ p_{t}(Z_{T})=z.\]
We may now draw analogies with the quantities we obtain in the MFG\ problem
for the separable case (\ref{payoff separable}).
Viewing, at the optimum,
the random variable $C\left( m_{T}^{\ast }\right) $ as a claim, the value
function $u$ in (\ref{u-DFN})\ as $w^{0}$ and inspecting equality (\ref%
{U-u-f}), we write, with some abuse of notation, %
\begin{equation}\label{U-u-indifference}
\begin{split}
u(x-&f(y,m,t),y,t)=\mathbb{E}\left[ \left. J(\mathcal{X}_{T}^{x-f(y,m,t),\ast
})\right \vert \mathcal{X}^\ast_{t}=x-f(y,m,t),Y_{t}=y\right]\\[1.5mm]
&=\mathbb{E}\left[ \left. J(X_{T}^{x,\ast }-C(m_{T}^{\ast })\right \vert
X^\ast_{t}=x,Y_{t}=y,m_{t}^{\ast }=m\right] =U(x,y,m,t).
\end{split}
\ee
In other words, we can think of   the mean field game value $U$ as the
value of the representative agent who starts with wealth $x$ and "writes" 
claim $C(m_{T}^{\ast })$. 
\vs
In turn, we interpret the mean field equilibrium optimal processes $X^{\ast }
$ and $\pi ^{\ast }$ as follows. At time $t,$ the representative agent
splits the initial $x$ into $x-f(y,m,t)$ and $f(y,m,t).$ With initial wealth 
$x-f\left( y,m,t\right) $, she implements the optimal policy $\alpha
^{x-f(y,m,t),\ast }$, given by (\ref{a-process-thm}) and rewritten below for
convenience,%
\bee\begin{split}
\alpha _{s}^{x-f(y,m,t),\ast
}&=c(Y_{s},s)H_{z}(H^{(-1)}(x-f(y,m,t),y,t)+L_{t,s},Y_{s},s)\\
&+H_{y}(H^{(-1)}(x-f(y,m,t),y,t)+L_{t,s},Y_{s},s).
\end{split}
\eee

This generates process $\mathcal{X}^{x-f(y,m,t),\ast}$ and, in particular, the random variable $\mathcal{X}_{T}^{x-f(y,m,t),\ast
}$ at terminal horizon $T$. 
\vs
In parallel, using the  initial wealth $f\left(
y,m,t\right) ,$ the representative agent follows strategy $C^{\ast }$ which
generates $C(m_{T}^{\ast })$ at $T.$ In other words, the "claim"\ $%
C(m_{T}^{\ast })$ admits the decomposition 
\[
C(m_{T}^{\ast })=f(y,m,t)+\int_{t}^{T}b(Y_{s},s)C_{s}^{\ast
}ds+\int_{t}^{T}C_{s}^{\ast }dW_{s}.
\]%
Therefore, the net wealth generated at $T$, starting at $x$, following
policies $\alpha^{x-f(y,m,t),\ast }$ and $C^{\ast },$ and
fulfilling the liability $C(m_{T}^{\ast })$ at  $T$ is given by%
\[
X_{T}^{x,\ast }-C(m_{T}^{\ast })
\]%
\[
=(x-f\left( y,m,t\right) )+f\left( y,m,t\right)
+\int_{t}^{T}b(Y_{s},s)\left( \alpha _{s}^{x-f(y,m,t),\ast }+C_{s}^{\ast
}\right) ds
\]%
\[
+\int_{t}^{T}(\alpha _{s}^{x-f(y,m,t),\ast }+C_{s}^{\ast
})dW_{s}-C(m_{T}^{\ast })=\mathcal{X}_{T}^{x-f(y,m,t),\ast },
\]
and this is precisely what the left hand side in (\ref{U-u-indifference})\
expresses. 
\vs

One may erroneously think that the above arguments trivialize the MFG\
problem, in the sense that one could "guess" (\ref{U-u-indifference}) from
starts and bypass the entire analysis. This is, of course, not correct
since, contrary to the traditional indifference/arbitrage-free valuation
problem, the claim in consideration $C(m_{T}^{\ast })$ is not a priori
given. Rather, $C(m_{T}^{\ast })$ is being created by the dynamic
interaction of the agents in $\left[ t,T\right] $ and has initial value $%
f(y,m,t)$ that needs to be found by solving equation (\ref{f-MFG-eqn}) and (%
\ref{pi-f}), or the autonomous equation (\ref{f-expo-mean}) in subsection~5.2.
\vs

Finally, we comment that a Feynma-Kac-type representation of the function $f$
solving \eqref{f-MFG-eqn} is consistent with the arbitrage-free price of the claim
$C(m^\ast_T)$ in analogy to \eqref{arbitrage}.

\section{Conclusions}

We revisited the classical optimal portfolio choice problem with partial
information and studied related mean field games of relative performance,
with couplings depending on the law of peers' wealth. For the former
problem, we introduced an alternative solution approach and produced
regularity and closed-form represenations for the value function and the
optimal processes, for general utilities. 
\vs
For the MFG games, we allowed for general dependence of the couplings on the
law of peers' wealth, and produced a master system, comprised by the
master equation and an optimality/compatibility condition of the mean field
equilibrium feedback policy. 
\vs

For both problems, we studied representative
cases and interpreted the corresponding solutions.   
\vs
There are several extensions of this work. Firstly, the methodology
developed herein may be extended to complete market models with general
factors, beyond reduced models generated from partial observations. Such
models incorporate, for example, local volatility, predictable returns, and
others, and have been studied by various authors for homothetic utilities;
see, among others,  \cite{Kim-Omberg}, \cite{Wachter-local} and \cite%
{Zariphopoulou-MOR} for homothetic utilities. 
\vs
Secondly, intermediate consumption and multi-assets may be, also,
incorporated. It is expected that for separable payoffs, analogous
connections with the single agent problem will hold, but the form of
equation (\ref{f-MFG-eqn}) is not obvious. 
\vs
Thirdly, the closed-form representations for the optimal wealth and
portfolio processes may be useful in studying the effects of collective
investment behavior on asset returns. For example, one may study the
so-called crowding risk (\cite{Baltas}, \cite{Barrosso et al}, \cite%
{Chincarini}) which arises when similar investment activity across many
agents ends up diluting value and performance.
\vs
In a different direction, one may study MFG\ with relative performance
defined in a semi-finite domain, for example in $[0,\infty ).$ In the
absence of competition, this case corresponds to non-negative wealth
constraints. Generally speaking, there are two distinct ways to incorporate
competition, multiplicatively and additively. The former was firstly modeled
in \cite{Lacker-Z.} for the case of power utilities, where the
effects of peer performance was formulated as the geometric average of the
players' wealth. For such couplings, the non-negativity constraint is
implicitly satisfied throughout, and the homotheticity of preferences
allowed for explicit solutions that led to random (contant in time)
equilibria. This specific case, power utilities and multiplicative-type
couplings, was subsequently studied by many authors, with representative
references being listed in the bibliography. However, the case of general
utility and/or multiplicative couplings remains open. If, on the other hand,
the competition is additive, as the case herein, one still expects a
solution to be as in \eqref{U-master-intro}, that is, 
\begin{equation}
U(x,t)=u(x-f(y,m,t),y,t)  \label{U-semi-definite}
\end{equation}%
where $u$ is the solution of the single agent problem in $\left[ 0,\infty
\right) \times \left[ 0,T\right] $ and $f>0$ solving an equation  like \eqref{f-intro}. 
However, the nonnegativity of the wealth argument in the
right-hand-side above will result in $U(x,t)$ defined in a time-evolving
(admissibility) domain of the form $x\geq f(y,m,t),$ $y\in \mathbb{R}$, $m\in \mathcal{P}_2$, %
and $t\in \left[ 0,T\right]$.  If the equation satisfied by $f$ has a unique,
smooth enough solution, the
admissibility domain can be fully characterized and representation \eqref
{U-semi-definite}  is expected to hold.
\vs
In both semi-definite domain cases, with multiplicative or additive
couplings, it would be interesting to study the emerging indifference
valuation problems which, however,  are expected to be non-standard. Indeed, it is not
clear what kind of "claim"\ multiplicative couplings would result to, at
mean field equilibrium. For the additive case, the difficulties are entirely
different, as questions related to "super-hedging", wealth range and the dynamic 
price of the coupling at the optimum may be challenging even in the complete market herein.
The authors are
currently working on both these directions.

\vs
Finally, incorporating idiosyncratic noise remains a challenging question,
as the linearization steps will not hold. The authors are currently working
on the exponential case and general couplings, allowing for individual
stocks in addition to the common one. 


\end{document}